\documentclass[11pt]{amsart}

\usepackage{a4wide, amsmath, amsfonts, amssymb, mathrsfs, amsthm}
\usepackage[hyperfootnotes=false]{hyperref}

\allowdisplaybreaks

\numberwithin{equation}{section}

\newtheorem{theorem}{Theorem}[section]
\newtheorem{lemma}[theorem]{Lemma}
\newtheorem{corollary}[theorem]{Corollary}
\newtheorem{proposition}[theorem]{Proposition}
\theoremstyle{remark}
\newtheorem{remark}[theorem]{Remark}

\newtheorem{definition}[theorem]{Definition}

\newcommand{\dd}{\,\mathrm{d}}
\newcommand{\R}{\mathbb{R}}
\newcommand{\C}{\mathbb{C}}
\newcommand{\N}{\mathbb{N}}
\newcommand{\T}{\mathbb{T}}
\newcommand{\Z}{\mathbb{Z}}
\newcommand{\1}{\mathbf{1}}
\renewcommand{\d}{\mathrm{d}}
\renewcommand{\epsilon}{\varepsilon}
\renewcommand{\phi}{\varphi}
\global\long\def\supp{\operatorname{supp}}
\global\long\def\F{\mathcal{F}}

\title[On Besov spaces with dominating mixed smoothness]{Characterization of Besov spaces with dominating mixed smoothness by differences}

\author[Nikolaev]{Paul Nikolaev}
\address{Paul Nikolaev, University of Mannheim, Germany}
\email{pnikolae@mail.uni-mannheim.de}

\author[Pr{\"o}mel]{David J. Pr{\"o}mel}
\address{David J. Pr{\"o}mel, University of Mannheim, Germany}
\email{proemel@uni-mannheim.de}

\author[Trabs]{Mathias Trabs}
\address{Mathias Trabs, Karlsruhe Institute of Technology, Germany}
\email{trabs@kit.edu}

\date{\today}

\begin{document}

\begin{abstract}
  Besov spaces with dominating mixed smoothness, on the product of the real line and the torus as well as bounded domains, are studied. A characterization of these function spaces in terms of differences is provided. Applications to random fields, like Gaussian fields and the stochastic heat equation, are discussed, based on a Kolmogorov criterion for Besov regularity with dominating mixed smoothness.
\end{abstract}

\maketitle 




\section{Introduction}

Regularity concepts for functions, related to Besov and Sobolev spaces with dominating mixed smoothness, play a fundamental role in many areas of pure and applied mathematics. For instance, they appear naturally in the context of approximation theory~\cite{Temlyakov1993}, numerical analysis~\cite{Schwab2008}, probability theory~\cite{Hummel2021}, statistics~\cite{GoldenshlugerLepsik2011} and the study of (stochastic) partial differential equations~\cite{Quer2007,CheminZhangZhang2015}. Function spaces with dominating mixed smoothness of Sobolev type were originally introduced by Nikol'ski\u{\i} \cite{Nikolsikii1962,Nikolsikii1963} in the sixties and afterwards generalized in numerous directions, like to the more general class of Besov spaces. For general introductions to Besov spaces with dominating mixed smoothness and their history we refer, e.g., to \cite{Schmeisser1987,Schmeisser2007,Triebel2019}.

Besov spaces with dominating mixed fractional smoothness are often classically defined by Fourier-analytic approaches relying, e.g., on an anisotropic generalizations of Littlewood--Paley theory, cf. e.g. \cite[Chapter~2]{Schmeisser1987}. While the Fourier-analytic definitions of Besov spaces with dominating mixed smoothness have, with no doubts, many merits, there has been a systematic on-going effort to derive various equivalent characterizations of the these function spaces, each coming with its own advantages. For instance, Besov spaces with dominating mixed smoothness can be equivalently defined by using atomic, wavelet and spline representations, see e.g. \cite[Chapter~1.1]{Triebel2019}, or characterized as tensor product of the corresponding isotropic function spaces, see \cite{Sickel2009}.

In the present paper, we define Besov spaces with dominating mixed smoothness, on the product of the real line and the torus, based on the anisotropic generalizations of the Littlewood--Paley theory, see \eqref{eq:Besov norm anisotropic} below. The main result are two equivalent characterizations of these function spaces in terms differences, see Definition~\ref{def: characterisation of Besov spaces} and Theorem~\ref{thm:characteristation}, assuming fractional regularity $\overline{\alpha}\in (0,1)^2$, integrability $\overline{p} \in[1,\infty]^2$ and $\overline{q} \in [1, \infty]^2$. Consequently, this equivalent characterization can be deduced for Besov spaces with dominating mixed smoothness on bounded domains. Similar results were previously derived, either on the whole space~$\R^d$, bounded domains or the torus~$\T^d$, by Triebel, Schmeisser and Ulrich, see e.g. \cite{Schmeisser1987,Ulrich2006}. Building on these works, we present an accessible and transparent proof, taking care of domains with a mixed structure.

The characterization by difference is well-suited to study the Besov regularity with dominating mixed smoothness of random fields, emerging in probability theory and related fields. Notice that, in contrast to their isotropic regularity, the Besov regularity of dominating mixed smoothness of many random fields does not get worse with increasing space dimension. For instance, the (fractional) Brownian sheet is, in higher dimension, much smoother when its regularity is measured with dominating mixed smoothness, see e.g. \cite{Kamont1996,Hummel2021}. This observation becomes essential, e.g., when developing pathwise approaches to stochastic partial differential equations, see e.g. \cite{Quer2007,Hairer2014}. Using the characterization by difference, we derive a Kolmogorov criterion for Besov regularity with dominating mixed smoothness of random fields. Consequently, we deduce the regularity of Gaussian fields and of the solution to the stochastic heat equation. 

\medskip

\noindent \textbf{Organization of the paper:} In Section~\ref{sec:Besov spaces} we recall the definition of Besov spaces with dominating mixed smoothness and present some of their elementary properties. Three equivalent characterizations of these Besov spaces are provided in Section~\ref{sec:integral characterisations}. The Besov regularity of random fields, such as Gaussian fields and the stochastic heat equation, is studied in Section~\ref{sec:regularity of random fields}.

\section{Besov spaces with dominating mixed smoothness}\label{sec:Besov spaces}

In this section we introduce Besov spaces with dominating mixed smoothness on the domain $\R \times \mathbb{T}$ with complex valued distributions, based on the corresponding anisotropic generalizations of the Littlewood--Paley theory. We shall present all results for Besov spaces with two-dimensional domains as this allows us to develop the main conceptional ideas and to avoid unnecessarily cumbersome notation. Nevertheless, let us remark that all presented results and their proofs can be modified to cover Besov spaces with higher dimensional domains as well as real valued distributions. Moreover, the extension of the presented results to vector valued distributions can be obtained by defining all operation componentwise.

\smallskip

We begin by setting up some notation. The Euclidean norm in $\R^d$ and the corresponding scalar product are denoted by $|\cdot |$ and $\langle\cdot,\cdot\rangle$, respectively. Analogously, \(\T^d\) is the d-dimensional torus, which we identify with the interval \([-\pi,\pi]^d\) in \(\R^d\). The space of Schwartz functions $f\colon\R^d\to\C$ is denoted by $\mathcal{S}(\R^{d}):=\mathcal{S}(\R^{d},\C)$ and its dual by $\mathcal{S}^{\prime}(\R^{d})$, which is the space of tempered distributions. By abusing the notation, we also denote by \(\langle \cdot, \cdot \rangle\) the dual paring between a Schwartz distribution and a Schwartz function. For $p\in [1,\infty]$ we write $L^p(\R^{d}):=L^p(\R^{d},\C)$ for the Lebesgue space of $p$-integrable functions on $\R^{d}$ with usual norm~$\|\cdot \|_p$. For a function $f\in L^{1}(\R^{d})$ the Fourier transform is defined by
\begin{equation*}
  \mathcal{F}f(\xi ):=\int_{\mathbb{R}^{d}}e^{-i\langle \xi,x\rangle}f(\xi)\dd x,\qquad \xi\in\R^d,
\end{equation*}
and so the inverse Fourier transform is given by $\mathcal{F}^{-1}f(x):=(2\pi)^{-d}\mathcal{F}f(-x)$ for $x\in\R^d$. If $f\in\mathcal{S}^{\prime}(\R^{d})$, then the usual generalization of the Fourier transform is considered. Moreover, we define the support of a distribution \(f \in \mathcal{S}'(\R^d)\) denoted by  \(\supp f\) as the intersection of all closed sets \(A\) with the property
\begin{equation*}
  \varphi \in C_c^\infty(\R^d), \; \; \supp \varphi \subseteq  \R^d \setminus  A  \Rightarrow \langle f, \varphi \rangle = 0 .
\end{equation*}

Furthermore, for measurable functions \(g_{i} \colon \R \times \T \to \C\), \(i=1,2\), \(h \colon \R \times \R \to \C\) we define the following two convolution operators by
\begin{equation*}
  g_1*h(x) = \int_{\R^2} g_1(x-y)h(y) \dd y
  \quad \text{and} \quad
  g_1*_{\pi_2} g_2(x) = \int_{\R \times \T} g_1(x-y)g_2(y) \dd y .
\end{equation*}

We will frequently use the notation $A_{\theta}\lesssim B_{\theta}$, for a generic parameter $\theta$, meaning that $A_{\theta}\le CB_{\theta}$ for some constant $C>0$ independent of $\theta$. We write $A_{\theta}\sim B_{\theta}$ if $A_{\theta}\lesssim B_{\theta}$ and $B_{\theta}\lesssim A_{\theta}$. For integers $j_{\theta},k_{\theta}\in\mathbb{Z}$ we write $j_{\theta}\lesssim k_{\theta}$ if there is some $N\in\N$ such that $j_{\theta}\le k_{\theta}+N$, and $j_{\theta}\sim k_{\theta}$ if $j_{\theta}\lesssim k_{\theta}$ and $k_{\theta}\lesssim j_{\theta}$.

\smallskip

Let us now recall the Littlewood--Paley characterization of isotropic Besov spaces. A dyadic partition of unity $(\chi,\rho)$ in dimension $d$ is given by two smooth functions on $\R^d$ satisfying $\supp\chi\subseteq \{x \in \R^d \, : \, |x| \le \frac{4}{3} \}$, $\supp\rho\subseteq\{x \in\R^d:\frac{3}{4}\le|x|\le\frac{8}{3}\}$ and $\chi(z)+\sum_{j\geq0}\rho(2^{-j}z)=1$ for all $z\in\R^d$. We set
\begin{equation*}
  \rho_{-1}:=\chi\quad\text{and}\quad\rho_{j}:=\rho(2^{-j}\cdot)\quad\text{for }j\ge0.
\end{equation*}
Taking a dyadic partition of unity $(\chi,\rho)$ in dimension two, the \textit{Littlewood--Paley blocks} are defined as
\begin{equation*}
  \Delta_{-1}f:=\mathcal{F}^{-1}(\rho_{-1}\mathcal{F}f)\quad\text{and}\quad\Delta_{j}f:=\mathcal{F}^{-1}(\rho_{j}\mathcal{F}f)\quad\text{for }j\geq0.
\end{equation*}
Note that, by the Paley--Wiener--Schwartz theorem (see e.g. \cite[Theorem~1.2.1]{Triebel1978}), $\Delta_{j}f$ is a smooth function for every $j\geq-1$ and for every $f\in\mathcal{S}^{\prime}(\R^{2})$ we have 
\begin{equation*}
  f=\sum_{j\geq-1}\Delta_{j}f:=\lim_{j\to\infty}S_{j}f\quad\text{with}\quad S_{j}f:=\sum_{i\leq j-1}\Delta_{i}f.
\end{equation*}
For $\alpha\in\mathbb{R}$ and $p,q\in(0,\infty]$ the \textit{isotropic Besov space} is defined as
\begin{align}\label{eq:isotropic Besov norm}
  &B_{p,q}^{\alpha}(\mathbb{R}^{2}):=\Big\{ f\in\mathcal{S}^{\prime}(\mathbb{R}^{2}): |f|_{\alpha,p,q}<\infty\Big\} \nonumber \\
  &\text{with}\quad
  |f|_{\alpha,p,q}:=\Big\|\big(2^{j\alpha}\|\Delta_{j}f\|_{p}\big)_{j\ge-1}\Big\|_{\ell^{q}}.
\end{align}
The class of Besov spaces contains as special cases the Sobolev--Slobodeckij spaces $B_{p,p}^{\alpha}(\mathbb{R}^{2})$, the H{\"o}lder spaces $B_{\infty,\infty}^{\alpha}(\mathbb{R}^{2})$ as well as the Nikolskii spaces $B_{p,\infty}^{\alpha}(\mathbb{R}^{2})$.

\smallskip

In order to introduce Besov spaces with dominating mixed smoothness, we need \textit{Lebesgue spaces of mixed integrability} $\overline{p}=(p_{1},p_{2})\in[1,\infty]^2$, which are given by
\begin{align*}
  &L^{\overline{p}}(\R \times \T) :=\big\{ f \colon \R \times \T \to \C \; \mathrm{is \; measurable \; and} \; \|f\|_{\overline{p}}<\infty\big\}\\
  &\text{with}\quad \|f\|_{\overline{p}} :=\Big(\int_{\R}\Big(\int_{\T}|f(x,y)|^{p_{2}}\dd y\Big)^{p_{1}/p_{2}}\dd x\Big)^{1/p_{1}}.
\end{align*}
We adopt the convention that we view functions \(f \in L^{\overline{p}}(\R \times \T)\) as periodic functions defined on \(\R^2\). This allows us to use the classical Fourier analysis without introducing trigonometrical polynomials. 

On the Lebesgue spaces with mixed integrability the standard inequalities like Minkowski's, H{\"o}lder's and Young's inequality are available.

\begin{lemma}\label{lem: anisotropic_inequalities}
  \;
  \begin{enumerate}
    \item[(i)] For \(\bar p \in[1,\infty]^2 \) and $f,g\in L^{\bar p}(\R \times \T) $ we have
      \begin{equation*}
        \|f+g\|_{\bar p} \le \|f\|_{\bar p} +  \|g\|_{\bar p}.
      \end{equation*}
    \item[(ii)] Let $\bar p,\bar p'\in [1,\infty]^2 $ be such that $1/p_{1}+1/p'_{1}=1$ and $1/p_{2}+1/p'_{2}=1$. If $f\in L^{\bar p}(\R \times \T )$ and $g\in L^{\bar p'}(\R \times \T)$, then $fg\in L^{(1,1)}(\R \times \T)$ and
     \begin{equation*}
       \|fg\|_{(1,1)}\le\|f\|_{\bar p}\|g\|_{\bar p'}.
    \end{equation*}
    \item[(iii)] Let $\bar p,\bar p',\bar r\in [1,\infty]^2 $ be such that \(r_{i} \ge p'_{i} \) and \(1+ 1/r_{i} =  1/p_{i} + 1/p'_{i}\) for \(i=1,2\). If $f\in L^{\bar p}(\R \times \T)$ and $g\in L^{\bar p'}(\R \times \T)$, then
      \begin{equation*}
        \|f*g\|_{\bar r} \le \|f\|_{\bar p} \|g\|_{\bar p'}.
      \end{equation*}
  \end{enumerate}
\end{lemma}

The proofs follow straight forward by applying the corresponding inequality for the classical Lebesgue spaces and are thus omitted here.

\smallskip

Let us introduce the space \(\mathcal{S}_{\pi_2}'(\R^2)\) of Schwartz distributions, which are periodic in the second component. More precisely,
\begin{equation*}
  \mathcal{S}_{\pi_2}'(\R^2)
  := \{ f \in \mathcal{S}'(\R^2) \colon f(x_1,x_2+2k\pi) = f(x_1,x_2) \; \text{for} \; \text{all} \; k \in \Z \},
\end{equation*}
where the equality in the above definition is understood in the sense of distributions. We understand the space of periodic Schwartz distribution as a subspace of Schwartz distributions. Following~\cite[Chapter 3.2.3]{Schmeisser1987} we can identify distributions  defined on \(\R \times \T\) with the space \(\mathcal{S}_{\pi_2}'(\R^2) \). The motivation for identifying functions with their periodic extensions on the entire space lies in the fact that it simplifies various tasks, such as the demonstration of maximal inequalities. As highlighted in~\cite[Remark 3.2.3 and Chapter 3.2.4]{Schmeisser1987}, this approach makes the many proofs more straightforward. A similar rationale applies to the Fourier transform: by considering the Fourier transform on the entire space, we can directly utilize well-established results, rather than working with Fourier coefficients and Fourier series. Ultimately, both representations are equivalent~\cite[Chapter~3.2.1-3.2.3]{Schmeisser1987}. For a dyadic partition of unity $(\chi,\rho)$ in dimension one, $f\in\mathcal{S}_{\pi_2}'(\R^{2})$, $\xi=(\xi_{1},\xi_{2})\in\R^{2}$ and $(j,k)\in\{-1,0,1,2,\dots\}^{2}$ the \emph{anisotropic Littlewood--Paley blocks} are given by
\begin{align*}
  & \Delta_{j}^{1}f:=\F^{-1}\big[\rho_{j}(\xi_{1})\F f(\xi)\big],\quad\Delta_{k}^{2}f:=\F^{-1}\big[\rho_{k}(\xi_{2})\F f(\xi)\big],\quad\Delta_{j,k}f:=\Delta_{j}^{1}\Delta_{k}^{2}f,\\
  & S_{j}^{1}f:=\sum_{i\le j-1}\Delta_{i}^{1}f,\quad S_{k}^{2}f:=\sum_{l\le k-1}\Delta_{j}^{2}f,\quad\text{and}\quad S_{j,k}f:=S_{j}^{1}S_{k}^{2}f,
\end{align*}
cf. \cite[Section~6.2.1]{Bahouri2011}. It will be convenience to write
\begin{equation*}
  \rho_{j}\otimes\rho_{k}(\xi):=\rho_{j}(\xi_{1})\rho_{k}(\xi_{2}),\quad\xi=(\xi_{1},\xi_{2})\in\R^{2}.
\end{equation*}
In particular, note that we have $\Delta_{j,k}f=\F^{-1}[\rho_{j}\otimes\rho_{k}]\ast f$ and
\begin{equation*}
  f=\sum_{j,k\ge-1}\Delta_{j,k}f:=\lim_{L\to\infty}S_{L}^{1}S_{L}^{2}f\quad\mbox{for any }f\in\mathcal{S}'(\R^{2}).
\end{equation*}
Indeed, for any $f\in\mathcal{S}_{\pi_2}'(\R^2)$ and $\phi\in\mathcal{S}(\R^{2})$ we have $\langle f-S_{L}^{1}S_{L}^{2}f,\phi\rangle=\langle f,\phi-S_{L}^{1}S_{L}^{2}\phi\rangle$. It thus suffices to verify $\phi=\lim_{L\to\infty}S_{L}^{1}S_{L}^{2}\phi$ for any Schwartz function $\phi\in\mathcal{S}(\R^{2})$. Since the Fourier transform is an automorphism on $\mathcal{S}(\R^{2})$, the convergence follows from $\lim_{L\to\infty}\sum_{j,k=-1}^{L}(\rho_{j}\otimes\rho_{k})\F\phi=\F\phi$.

Following~\cite[Chapter~2.2]{Schmeisser1987} we define, for parameters
\begin{equation*}
  \overline{\alpha}:=(\alpha_{1},\alpha_{2})\in\R^{2}, \quad \overline{p}=(p_{1},p_{2})\in[1,\infty]^{2}\quad \text{and}\quad \overline{q}=(q_{1},q_{2})\in[1,\infty],
\end{equation*}
the \emph{Besov spaces with dominating mixed smoothness} as
\begin{align}\label{eq:Besov norm anisotropic}
  B_{\bar{p},\bar{q}}^{\bar{\alpha}}(\R \times \T) & :=\Big\{ f\in\mathcal{S}_{\pi_2}'(\R^2):\|f\|_{\overline{\alpha},\overline{p},\overline{q}}<\infty\Big\}\quad\text{with}\quad\nonumber \\
  \|f\|_{\overline{\alpha},\overline{p},\overline{q}} & :=\Big\|\Big(2^{\alpha_{1}j}\big\|\big(2^{\alpha_{2}k}\|\Delta_{j,k}f\|_{\bar p}\big)_{k\ge-1}\big\|_{\ell^{q_{2}}}\Big)_{j\ge-1}\Big\|_{\ell^{q_{1}}}.
\end{align}
For \(f \in \mathcal{S}_{\pi_2}'(\R^2)\) we notice that the dyadic block $\Delta_{j,k}f=\F^{-1}[\rho_{j}\otimes\rho_{k}]\ast f$ are periodic in the second component and therefore it makes sense to take the \(L^{\overline{p}}\)-norm in the second component over the torus \(\T\). Moreover, we can identify \(f \in L^{\bar p}(\R \times \T)\) as a periodic distribution in the canonical way:
\begin{lemma}
  For any \(f \in L^{\bar p}(\R \times \T)\) the distribution
\begin{equation*}
  \mathcal{S}(\R^2)\ni\phi\mapsto\langle f,\varphi \rangle = \int_{\R^2} f(x_1,x_2) \varphi(x_1,x_2) \dd x_2 \dd x_1
\end{equation*}
  is well-defined in $\mathcal{S}_{\pi_2}'(\R^2)$, where $f$ is periodically extended on \(\R \times \R\).
\end{lemma}
\begin{proof}
We need to show that the integral is finite. We start by deploying the monotone convergence to obtain 
\begin{align}\label{eq: lp_torus_distr}
  &\int_{\R^2}| f(x_1,x_2) \varphi(x_1,x_2)| \dd x_2 \dd x_1 \nonumber \\
  &\quad= \lim \limits_{m \to \infty} \int_{\R \times [-\pi (2m+1), \pi (2m+1)] } |f(x_1,x_2) \varphi(x_1,x_2)| \dd x_2 \dd x_1 \nonumber \\
  &\quad\le \| f \varphi\|_{(1,1)} + \lim \limits_{m \to \infty} \int_{\R \times \{ [-\pi (2m+1), -\pi ] \cup [\pi, \pi (2m+1)]  \} }| f(x_1,x_2) \varphi(x_1,x_2)| \dd x_2 \dd x_1.
\end{align}
The first term is finite by H{\"o}lder's inequality. For the right integral of the last term we obtain
\begin{align*}
  &\; \int_{\R \times [\pi,\pi (2m+1)]  }| f(x_1,x_2) \varphi(x_1,x_2)| \dd x_2 \dd x_1 \\
  &\quad=\; \sum\limits_{k=1}^m  \int_{\R} \int\limits_{\pi (2k-1)}^{\pi (2k+1)}  |f(x_1,x_2) \varphi(x_1,x_2)| \dd x_2\dd x_1 \\
  &\quad\le \;  \sum\limits_{k=1}^m \int_{\R} \bigg(\int\limits_{\pi (2k-1)}^{\pi (2k+1)}  |f(x_1,x_2)|^{p_2} \dd x_2   \bigg)^{\frac{1}{p_2}} \bigg(\int\limits_{\pi (2k-1)}^{\pi (2k+1)}  |\varphi(x_1,x_2)|^{q_2} \dd x_2    \bigg)^{\frac{1}{q_2}} \dd x_1  \\
  &\quad=\;  \sum\limits_{k=1}^m   \int_{\R} \bigg(\int_{\T}  |f(x_1,x_2)|^{p_2} \dd x_2   \bigg)^{\frac{1}{p_2}} \bigg(\int\limits_{\pi (2k-1) }^{\pi (2k+1)}  |\varphi(x_1,x_2)|^{q_2} \dd x_2   \bigg)^{\frac{1}{q_2}} \dd x_1 \\
  &\quad\le\;  \|f\|_{\bar p}  \sum\limits_{k=1}^m  \bigg( \int_{\R} \bigg(\int\limits_{\pi (2k-1)}^{\pi (2k+1)}  |\varphi(x_1,x_2)|^{q_2} \dd x_2    \bigg)^{\frac{q_1}{q_2}} \dd x_1 \bigg)^{\frac{1}{q_1}},
\end{align*}
where we used H{\"o}lder's inequality with \(1/p_{i} + 1/q_{i} =1 \) for \(i=1,2\). It remains to show that the last integral is summable over~\(k\). Using the fact the \(\varphi\) decays faster than any polynomial we find
\begin{align*}
  &\bigg( \int_{\R} \bigg(\int\limits_{\pi(2k-1)}^{\pi  (2k+1)}  |\varphi(x_1,x_2)|^{q_2}   \dd x_2  \bigg)^{\frac{q_1}{q_2}} \dd x_1 \bigg)^{\frac{1}{q_1}} \\
  &\,\le   \bigg( \int_{\R} \bigg(\int\limits_{\pi (2k-1)}^{\pi (2k+1)} (1+x_1)^{-\frac{3q_2}{q_1}} x_2^{ -3q_2-1} \sup\limits_{x_1,x_2 \in \R}  \big( |\varphi(x_1,x_2) |^{q_2} (1+x_1)^{\frac{3q_2}{q_1}} x_2^{3q_2+1} \big)  \dd x_2  \bigg)^{\frac{q_1}{q_2}}  \dd x_1 \bigg)^{\frac{1}{q_1}} \\
  &\,\le   \bigg( \int_{\R}  (1+x_1)^{-3} \bigg(\int\limits_{\pi (2k-1)}^{\pi (2k+1)}  x_2^{-3q_2-1} \dd x_2   \bigg)^{\frac{q_1}{q_2}} \dd x_1 \bigg)^{\frac{1}{q_1}} \sup\limits_{x_1,x_2 \in \R} \big(|\varphi(x_1,x_2) |^{q_2} (1+x_1)^{\frac{3}{q_1}} x_2^{3+ \frac{1}{q_2}}\big) \\
  &\,\lesssim   (2k-1)^{-2}  \sup\limits_{x_1,x_2 \in \R} \big(|\varphi(x_1,x_2) |^{q_2} (1+x_1)^{\frac{3}{q_1	}} x_2^{3+\frac{1}{q_2}}\big).
\end{align*} 
Hence, we find 
\begin{align*}
  &\int_{\R \times [\pi,\pi (2m+1)]  } f(x_1,x_2) \varphi(x_1,x_2) \dd x_2 \dd x_1 \\
  &\quad\lesssim \sup\limits_{x_1,x_2 \in \R} \big( |\varphi(x_1,x_2) |^{q_2} (1+x_1)^{\frac{3}{q_1	}} x_2^{3+\frac{1}{q_2}} \big) \sum\limits_{k=1}^\infty \frac{1}{(2k-1)^2}
  < \infty.
\end{align*}
A similar argument can be used to bound the remaining integral in~\eqref{eq: lp_torus_distr}. Hence, we have shown the continuity of~\(f\) as a Schwartz distribution. The periodicity in the second component is obvious and therefore we obtain \(f \in \mathcal{S}_{\pi_2}'(\R^2)\).
\end{proof}

This allows us to define the Fourier transformation of a function \(f \in L^{\bar p}(\R \times \T)\) by identifying the function \(f\) as a periodic distribution in the second component. Hence, we have
\begin{equation*}
  \langle \mathcal{F}(f),\varphi \rangle  = \langle f, \mathcal{F}(\varphi) \rangle
\end{equation*}
for \(f \in L^{\bar p}(\R \times \T), \; \varphi \in \mathcal{S}(\R^2)\). We thus understand the support of \(\mathcal{F}(f)\) in the sense of distributions.

\smallskip

The following Bernstein type inequalities have been proved in \cite{Paicu2005,Chemin2007}. In fact, the proof is analogous to the isotropic case.

\begin{lemma}\label{lem:bernstein inequalities}
  Let $\mathcal{B}$ be a ball in $\R$ and $\mathcal{C}$ be an annulus in $\R$. Let $\alpha,\beta\in\R$, $1\le p'_1\le p_{1}\le\infty$ and $1\le p_2'\le p_2\le\infty$. For any $f\in C^{\infty}(\R \times \T) $ we have
  \begin{enumerate}
    \item[(i)] If $f\in  L^{(p'_1,p_2)}(\R \times \T)$, $\supp\F f\subseteq2^{j}\mathcal{B}\times\R$, then $\|\partial_{1}^{\alpha}f\|_{(p_{1},p_{2})}\lesssim 2^{j(|\alpha|+\frac{1}{p'_{1}}-\frac{1}{p_{1}})}\|f\|_{(p'_{1},p_{2})}$.
    \item[(ii)] If $f\in  L^{(p_1,p'_2)}(\R \times \T)$, $\supp\F f\subseteq\R\times 2^{k}\mathcal{B}$, then $\|\partial_{2}^{\beta}f\|_{(p_{1},p_{2})}\lesssim 2^{k(|\beta|+\frac{1}{p'_{2}}-\frac{1}{p_{2}})}\|f\|_{(p_{1},p'_{2})}$.
    \item[(iii)] If $f\in  L^{\bar p}(\R \times \T)$, $\supp\F f\subseteq2^{j}\mathcal{\mathcal{C}}\times\R$, then $\|f\|_{\bar p}\lesssim 2^{-jN}\|\partial_{1}^{N}f\|_{\bar p}$.
    \item[(iv)] If $f\in  L^{\bar p}(\R \times \T)$, $\supp\F f\subseteq\R\times2^{k}\mathcal{C}$, then $\|f\|_{\bar p}\lesssim 2^{-kN}\|\partial_{2}^{N}f\|_{\bar p}$.
  \end{enumerate}
\end{lemma}
\begin{proof}
  For sake of completeness we will prove the second assertion. The other claims follow by a similar argument. Let \(\psi \in \mathcal{S}(\R) \) have compact Fourier support and \(\mathcal{F}(\psi)(\xi_2) = 1 \) for \(\xi_2 \in \mathcal{B}\). Define \(\psi_k(y_2) := 2^k \psi(2^k y_2)\). Then \(\mathcal{F}(\psi_k)(\xi_2) = \mathcal{F}(\psi)(\xi_2/2^k) \) and \( \mathcal{F} (\psi_k)(\xi_2) = 1 \) for \(\xi_2 \in 2^{k}\mathcal{B} \). Furthermore, since \(f\) corresponds to a smooth function once can use dominated convergence theorem to show that \(\partial_2^\beta f\) has Fourier support contained in the same set. Consequently, for \(\varphi \in \mathcal{S}(\R^2)\), then
  \begin{align*}
    &\int_{\R^2}  \partial_2^\beta f(\xi_1,\xi_2) \mathcal{F}(\varphi)(\xi_1,\xi_2) \dd \xi_2 \dd \xi_1 \\
    &\quad=\;   \int_{\R^2}  \mathcal{F}( \varphi \mathcal{F}(\psi_k)(\cdot_2) )(\xi_1,\xi_2)  \partial_2^\beta f(\xi_1,\xi_2) \dd \xi_2 \dd \xi_1  \\
    &\quad=\; \int_{\R^2}  \int_{\R^2} e^{-i(\xi_1x_1+\xi_2x_2)}  \varphi(x_1,x_2) \mathcal{F}(\psi_k)(x_2)
    \partial_2^\beta f(\xi_1,\xi_2) \dd x_2 \dd x_1 \dd \xi_2 \dd \xi_1  \\
    &\quad=\; \int_{\R^2}  \int_{\R^2} \int_{\R} e^{-i(\xi_1x_1+\xi_2x_2)} e^{-ix_2y_2} \psi_k(y_2) \varphi(x_1,x_2)
    \partial_2^\beta f(\xi_1,\xi_2) \dd y_2 \dd x_2 \dd x_1 \dd \xi_2 \dd \xi_1  \\
    &\quad=\; \int_{\R^2}  \int_{\R^2} \int_{\R} e^{-i(\xi_1x_1+\xi_2x_2)} e^{-ix_2y_2} e^{i x_2 \xi_2} \psi_k(y_2-\xi_2) \varphi(x_1,x_2)
    \partial_2^\beta f(\xi_1,\xi_2) \dd y_2 \dd x_2 \dd x_1 \dd \xi_2 \dd \xi_1  \\
    &\quad=\; \int_{\R^2}  \int_{\R^2} e^{-i(\xi_1x_1+x_2y_2)}  \varphi(x_1,x_2) \dd x_2 \dd x_1 \int_{\R}  \psi_k(y_2-\xi_2)
    \partial_2^\beta f(\xi_1,\xi_2)  \dd \xi_2 \dd y_2  \dd \xi_1  \\
    &\quad=\; \int_{\R^2}  \mathcal{F}(\varphi)(\xi_1,y_2) \int_{\R} \partial_2^\beta  \psi_k(y_2-\xi_2)
    f(\xi_1,\xi_2)  \dd \xi_2 \dd y_2  \dd \xi_1 ,
  \end{align*}
  which implies
  \begin{equation*}
    \partial_2^\beta f(x_1,x_2) =  \int_{\R} \partial_2^\beta \psi(x_2-y_2) f(x_1,y_2) \dd y_2
  \end{equation*}
  by the fundamental lemma of calculus of variatuion and since the Fourier transformation is an automorphism on \(\mathcal{S}(\R^2)\). Applying Young's inequality we find
  \begin{align*}
    \| \partial_2^\beta f \|_{(p_1,p_2)} &\le    \bigg( \int_{\R} | \partial_2^\beta \psi_k(y_2)|^{r_2} \dd y_2 \bigg)^{\frac{1}{{r_2}}}   \|f\|_{(p_1,p'_2)}  \\
    &= 2^{k(1+|\beta|)} \bigg( \int_{\R} | (\partial_2^\beta \psi)(2^ky_2)|^{r_2} \dd y_2 \bigg)^{\frac{1}{{r_2}}}    \|f\|_{(p_1,p'_2)} \\
    &= 2^{k(1+|\beta|-1/{r_2})} \bigg( \int_{\R} | (\partial_2^\beta \psi)(y_2)|^{r_2} \dd y_2 \bigg)^{\frac{1}{{r_2}}}    \|f\|_{(p_1,p'_2)} \\
    &= 2^{k(|\beta|+\frac{1}{p'_{2}}-\frac{1}{p_{2}})} \bigg( \int_{\R} | \partial_2^\beta \psi(y_2)|^{r_2} \dd y_2 \bigg)^{\frac{1}{{r_2}}}    \|f\|_{(p_1,p'_2)} \\
  \end{align*}
  with \(\frac{1}{r_2} = 1- \frac{1}{p'_2} + \frac{1}{p_2}\). The claim follows from the fact that \(\psi \in \mathcal{S}(\R)\) and therefore every derivative is \(L^{r_2}\) integrable	.
\end{proof}

\subsection{Elementary properties}

As immediate corollaries of Lemma~\ref{lem:bernstein inequalities} and the definition of the Besov spaces with dominating mixed smoothness, we obtain the following embedding results and the subsequent lifting property.

\begin{lemma}\label{lem:embeddings}
  For $i=1,2$ and $j=1,2,3$ let $\alpha_{i},\beta_{i}\in\R$ with $\beta_{i}<\alpha_{i}$, $1 \le p_{j},q_{j}$ with $\overline{p}=(p_{1},p_{2}) \in [1, \infty]^2$ and $\overline{q}=(q_{1},q_{2}) \in [1, \infty]^2$. The following are continuous embeddings:
  \begin{enumerate}
    \item[(i)] $B_{\bar{p},(q_{3},q_{2})}^{(\alpha_{1},\alpha_{2})}(\R \times \T )\subseteq B_{\bar{p},(q_{1},q_{2})}^{(\beta_{1},\alpha_{2})}(\R \times \T)$ and $B_{\bar{p},(q_{1},q_{3})}^{(\alpha_{1},\alpha_{2})}(\R \times \T)\subseteq B_{\bar{p},(q_{1},q_{2})}^{(\alpha_{1},\beta_{2})}(\R \times \T )$,
    \item[(ii)] $B_{\bar{p},\bar{q}}^{(\alpha_{1},\alpha_{2})}(\R \times \T )\subseteq L^{\bar p}(\R \times \T )$ for any $(\alpha_{1},\alpha_{2})\in(0,\infty)^{2}$,
    \item[(iii)] $B_{(p_{3},p_{2}),\bar{q}}^{(\alpha_{1}+(\frac{1}{p_{3}}-\frac{1}{p_{1}}),\alpha_{2})}(\R \times \T)\subseteq B_{(p_{1},p_{2}),\bar{q}}^{(\alpha_{1},\alpha_{2})}(\R \times \T)$ if $p_{3}\le p_{1}$,
    \item[(iv)] $B_{(p_{1},p_{3}),\bar{q}}^{(\alpha_{1},\alpha_{2}+(\frac{1}{p_{3}}-\frac{1}{p_{2}}))}(\R \times \T)\subseteq B_{(p_{1},p_{2}),\bar{q}}^{(\alpha_{1},\alpha_{2})}(\R \times \T)$ if $p_{3}\le p_{2}$.
  \end{enumerate}
\end{lemma}

\begin{proof}
  (i): First, we observe that $(2^{j(\beta_{1}-\alpha_{1})})_{j\ge-1}\in \ell^{(1-1/q_3)^{-1}}$, which implies
  \begin{equation*}
    B_{\bar{p},(q_{3},q_{2})}^{(\alpha_{1},\alpha_{2})}(\R \times \T)\subseteq B_{\bar{p},(1,q_{2})}^{(\beta_{1},\alpha_{2})}(\R \times \T).
  \end{equation*}
  Moreover, the monotonicity of the \(l^q\) spaces, i.e. \(l^q \subseteq l^{\widetilde{q}}\) for \(1\le q \le \widetilde{q}\), proves the first embedding. The second embedding can be obtained analogously.
 
  (ii): Since $\|f\|_{\bar{p}}\le\sum_{j,k}\|\Delta_{j,k}f\|_{\bar{p}}=\|f\|_{(0,0),(p_{1},p_{2}),(1,1)}$ for any $\bar{p}\in[1,\infty]^{2}$, we deduce from~(i) that $B_{\bar{p},\bar{q}}^{\bar{\alpha}}(\R \times \T)\subseteq B_{\bar{p},(1,1)}^{(0,0)}(\R \times \T)\subseteq L^{\overline{p}}(\R \times \T)$.

  (iii) and (iv): Note that due to $\supp\F(\Delta_{j,k}f)\subseteq2^{j}\mathcal{B}\times2^{k}\mathcal{B}$ and $\Delta_{j,k}f $ smooth, Lemma~\ref{lem:bernstein inequalities} (i) and (ii) yield
  \begin{equation*}
    \|\Delta_{j,k}f\|_{(p_{1},p_{2})}=2^{j(\frac{1}{p_{3}}-\frac{1}{p_{1}})}\|\Delta_{j,k}f\|_{(p_{3},p_{2})}\quad\mbox{and}\quad\|\Delta_{j,k}f\|_{(p_{1},p_{2})}\lesssim2^{k(\frac{1}{p_{3}}-\frac{1}{p_{2}})}\|\Delta_{j,k}f\|_{(p_{1},p_{3})}.\qedhere
  \end{equation*}
\end{proof}

\begin{lemma}
  Let $\overline{\alpha}=(\alpha_{1},\alpha_{2})\in\R^2$, $\overline{p} = (p_{1},p_{2}) \in[1,\infty]^2$ and $\overline{q} =(q_1,q_2) \in [1, \infty]^2$. If $f\in B_{\overline{p},\overline{q}}^{\overline{\alpha}}(\R \times \T)$, then $\partial_{x_1}f\in B_{\overline{p},\overline{q}}^{(\alpha_{1}-1,\alpha_{2})}(\R \times \T)$ and $\partial_{x_2}f\in B_{\overline{p},\overline{q}}^{(\alpha_{1},\alpha_{2}-1)}(\R \times \T)$.
\end{lemma}

\begin{proof}
  Let $f$ be a smooth function with compact support and recall
  \begin{equation*}
    \|\partial_{x_1}f\|_{(\alpha_1-1,\alpha_2),\overline{p},\overline{q}}
    =\Big\|\Big(2^{(\alpha_{1}-1)j}\big\|\big(2^{\alpha_{2}k}\|\Delta_{j,k}(\partial_{x_1}f)\|_{p_{1},p_{2}}\big)_{k\ge-1}\big\|_{\ell^{q_{2}}}\Big)_{j\ge-1}\Big\|_{\ell^{q_{1}}}.
  \end{equation*}
  First note that
  \begin{align*}
    \Delta_{j,k}(\partial_{x_1}f)
    &= \mathcal{F}^{-1} \big(\rho_j(\xi_1)\rho_k(\xi_2)\mathcal{F}(\partial_{x_1}f)\big)\\
    &= \mathcal{F}^{-1} (\rho_j(\xi_1)\rho_k(\xi_2))*(\partial_{x_1}f)
    = \partial_{x_1} \big(\mathcal{F}^{-1} (\rho_j(\xi_1)\rho_k(\xi_2))*f\big)
    =\partial_{x_1}  \Delta_{j,k}(f),
  \end{align*}
  where we used the properties of the convolution operator. Since $\supp  \mathcal{F} ( \Delta_{j,k}(f) ) \subseteq 2^j B \times \R$ for some ball $B\subseteq \R$, Lemma~\ref{lem:bernstein inequalities}~(i) yields
  \begin{equation*}
    \|\Delta_{j,k}(\partial_{x_1}f)\|_{p_{1},p_{2}} = \|\partial_{x_1}\Delta_{j,k}(f)\|_{p_{1},p_{2}}
    \lesssim 2^j \| \Delta_{j,k}(f)\|_{p_{1},p_{2}}.
  \end{equation*}
  Hence, we get
  \begin{equation*}
    \|\partial_{x_1}f\|_{(\alpha_1-1,\alpha_2),\overline{p},\overline{q}}
    \lesssim \Big\|\Big(2^{\alpha_{1}j}\big\|\big(2^{\alpha_{2}k}\|\Delta_{j,k}(f)\|_{p_{1},p_{2}}\big)_{k\ge-1}\big\|_{\ell^{q_{2}}}\Big)_{j\ge-1}\Big\|_{\ell^{q_{1}}}
    =  \| f\|_{(\alpha_1,\alpha_2),\overline{p},\overline{q}}
  \end{equation*}
  and thus $f\mapsto \partial_{x_1}f$ is a bounded and linear operator on $C_c^\infty$, which has a continuous extension from $C_c^\infty(\R \times \T)$ to $B_{\overline{p},\overline{q}}^{\overline{\alpha}}(\R \times \T)$.
\end{proof}

For the regularity analysis of the Gaussian fields in Section~\ref{sec:regularity of random fields}, the following generalisation of Young's inequality for convolutions will be useful.

\begin{proposition}
  Let $\alpha_{1},\alpha_{2},\beta_{1},\beta_{2}\in\R$ and $p_{1},p_{2},q_{1},q_{2}\in[1,\infty]$ satisfying
  \begin{equation*}
    0\le\frac{1}{p}:=\frac{1}{p_{1}}+\frac{1}{p_{2}}-1\le 1
    \quad\text{and}\quad
    0\le\frac{1}{q}:=\frac{1}{q_{1}}+\frac{1}{q_{2}}-1\le 1.
  \end{equation*}
  Then, for any $f\in B_{(p_{1},p_{1}),(q_{1},q_{1})}^{(\alpha_{1},\alpha_{2})}(\R \times \T)$ and $g\in B_{(p_{2},p_{2}),(q_{2},q_{2})}^{(\beta_{1},\beta_{2})}(\R \times \T)$ we have
  \begin{equation*}
    f\ast_{\pi_2} g\in B_{(p,p),(q,q)}^{(\alpha_{1}+\beta_{1},\alpha_{2}+\beta_{2})}(\R \times \T)
  \end{equation*}
  with
  \begin{equation*}
    \|f\ast_{\pi_2} g\|_{(\alpha_{1}+\beta_{1},\alpha_{2}+\beta_{2}),(p,p),(q,q)}
    \lesssim\|f\|_{(\alpha_{1},\alpha_{2}),(p_{1},p_{1}),(q_{1},q_{1})}
    \|g\|_{(\beta_{1},\beta_{2}),(p_{2},p_{2}),(q_{2},q_{2})}.
  \end{equation*}
\end{proposition}

\begin{proof}
  The Littlewood--Paley blocks of the convolution satisfy for $j,k\ge-1$
  \begin{equation*}
    \Delta_{j,k}(f\ast_{\pi_2} g)
    =\F^{-1}\big[(\rho_{j}\otimes\rho_{k})\F f\F g\big]
    =\F^{-1}[(\rho_{j}\otimes\rho_{k})^{1/2}\F f]\ast\F^{-1}[(\rho_{j}\otimes\rho_{k})^{1/2}\F g].
  \end{equation*}
  Using Young's inequality for convolutions, we bound
  \begin{align*}
    2^{j(\alpha_{1}+\beta_{1})+k(\alpha_{2}+\beta_{2})}\|\Delta_{j,k}(f\ast_{\pi_2} g)\|_{(p,p)}
    \leq & \big(2^{j\alpha_{1}+k\alpha_{2}}\|\F^{-1}[(\rho_{j}\otimes\rho_{k})^{1/2}\F f]\|_{p_{1}}\big)\\
    & \times\big(2^{j\beta_{1}+k\beta_{2}}\|\F^{-1}[(\rho_{j}\otimes\rho_{k})^{1/2}\F g]\|_{p_{2}}\big).
  \end{align*}
  Hence, by the H{\"o}lder's inequality it suffices to show
  \begin{equation}\label{eq:PrConv}
    \Big\|\Big(2^{j\alpha_{1}} \big\|\big(2^{k\alpha_{2}} \|\F^{-1}[(\rho_{j}\otimes\rho_{k})^{1/2}\F f]\|_{p_{1}}\big)_{{k\geq -1}}\big\|_{\ell^{q_1}}\Big)_{
    _{j\geq -1}}\Big\|_{\ell^{q_2}}
    \lesssim\|f\|_{(\alpha_{1},\alpha_{2}),(p_{1},p_1),(q_1,q_2)}
  \end{equation}
  (and consequently the analogous estimate holds true for $g$). To verify \eqref{eq:PrConv}, we decompose $f=\sum_{j,k}\Delta_{j,k}f$. Due to the compact support of $\rho$ and Young's inequality we obtain
  \begin{align*}
    & 2^{j\alpha_{1}+k\alpha_{2}}\|\F^{-1}[(\rho_{j}\otimes\rho_{k})^{1/2}\F f]\|_{p_{1}}\\
    & \quad\le2^{j\alpha_{1}+k\alpha_{2}}\sum_{j',k'}\big\|\F^{-1}\big[(\rho_{j}\otimes\rho_{k})^{1/2}\F[\Delta_{j',k'}f]\big]\big\|_{p_{1}}\\
    & \quad\le 2^{j\alpha_{1}+k\alpha_{2}}\sum_{j'=j-1}^{j+1} \sum_{k' = k-1}^{k+1}\|\F^{-1}[(\rho_{j}\otimes\rho_{k})^{1/2}]\|_{1}\|\Delta_{j',k'}f\|_{p_{1}}\\
    & \quad\lesssim \sum_{j'=j-1}^{j+1} \sum_{k' = k-1}^{k+1} 2^{-(j'-j)\alpha_{1}}  2^{-(k'-k)\alpha_{2}} \big(2^{j'\alpha_{1}+k'\alpha_{2}}\|\Delta_{j',k'}f\|_{p_{1}}\big) \\
    & \quad\lesssim \sum_{j'=j-1}^{j+1} \sum_{k' = k-1}^{k+1} 2^{j\alpha_{1}+k\alpha_{2}} \|\Delta_{j',k'}f\|_{p_{1}} .
  \end{align*}
  Taking the $\ell^{q_1}$- and $\ell^{q_2}$-norm, we conclude
  \begin{align*}
   \Big\|\Big(2^{j\alpha_{1}} \big\|\big(2^{k\alpha_{2}} \|\F^{-1}[(\rho_{j}\otimes\rho_{k})^{1/2}\F f]\|_{p_{1}}\big)_{{k\geq -1}}\big\|_{\ell^{q_1}}\Big)_{
    _{j\geq -1}}\Big\|_{\ell^{q_2}}
   &\lesssim \sum_{j=-1}^{\infty} \sum_{k =- 1}^{\infty} 2^{j\alpha_{1}+k\alpha_{2}} \|\Delta_{j,k}f\|_{p_{1}} \\
    &=\|f\|_{(\alpha_{1},\alpha_{2}),(p_{1},p_{1}),(q_1,q_2)}.\qedhere
  \end{align*}
\end{proof}

\section{Characterization via differences}\label{sec:integral characterisations}

In this section we derive two equivalent integral representations of the Besov norm with dominating mixed smoothness from \eqref{eq:Besov norm anisotropic}. To give the classical background, let us recall corresponding characterization for isotropic Besov spaces. It is well-known that on the isotropic Besov spaces $B^\alpha_{p,q}(\R^2)$, for \(\alpha \in (0,1)\) and \(p,q \ge 1 \),  the norm $|f|_{\alpha,p,q}$ from \eqref{eq:isotropic Besov norm} and the two norms
\begin{align*}
  |f|_{\alpha,p,q}^{(1)}&:= \|f\|_{p} + \Big( \int_{\R} |h|^{-\alpha q } \sup\limits_{|r| \le |h|} \|f(\cdot+r) -f(\cdot) \|_{p}^q \frac{\dd h}{|h|^2}  \Big)^{1/q},\\
  |f|_{\alpha,p,q}^{(2)}&:= \|f\|_{p} + \Big( \int_{\R} |h|^{-\alpha q }  \|f(\cdot+r) -f(\cdot) \|_{p}^q \frac{\dd h}{|h|^2}  \Big)^{1/q}
\end{align*}
are all equivalent, for \(f \in B_{p,q}^{\alpha}(\mathbb{R}^{2})\), see e.g. \cite[Theorem~2.5.12]{Triebel1978}.

\smallskip

To obtain the analogous result in the case of Besov spaces with dominating mixed smoothness, we define the rectangular increments
\begin{equation*}
  \square_{h}f(x):=f(x_{1}+h_{1},x_{2}+h_{2})-f(x_{1},x_{2}+h_{2})-f(x_{1}+h_{1},x_{2})+f(x_{1},x_{2})
\end{equation*}
and the directional increments 
\begin{align*}
  \delta_{h_1}f(x) &:= f(x_1+h_1,x_2)-f(x_1,x_2), \\
  \delta_{h_2}f(x) &:= f(x_1,x_2+h_2)-f(x_1,x_2)
\end{align*}
for $h=(h_{1},h_{2})$, $x=(x_{1},x_{2})\in\R^2$.

\begin{definition}\label{def: characterisation of Besov spaces}
  For $\overline{\alpha}=(\alpha_{1},\alpha_{2})\in(0,1)^{2}$, $\overline{p}=(p_{1},p_{2})\in[1,\infty]^{2}$ and $\overline{q}=(q_{1},q_{2})\in[1,\infty]^2$ we define
  \begin{align*}
    \|f\|_{\overline{\alpha},\overline{p},\overline{q}}^{(1)}  :=& \, \|f\|_{\overline{p}}  +  \Big( \int_{\R} h_{1}^{-\alpha_{1}q_{1}}  \sup\limits_{|r_1| \le |h_1|} \|\delta_{r_{1}}^1 f\|_{{\overline{p}}}^{q_{1}}\frac{\dd h_{1}}{|h_{1}|}  \Big)^{\frac{1}{q_1}}  + \Big( \int_{\T} h_2^{-\alpha_{2}q_{2}}
    \sup\limits_{|r_2| \le |h_2|} \|\delta_{r_{2}}^2 f\|_{\overline{p}}^{q_{2}}\frac{\dd h_{2}}{|h_{2}|}\Big)^{\frac{1}{q_2}}  \\
    &+ \Big(\int_{\R}|h_{1}|^{-\alpha_{1}q_{1}}\Big(\int_{\T}|h_{2}|^{-\alpha_{2}q_{2}}\sup_{|r_{1}|\le|h_{1}|,|r_{2}|\le h_{2}}\|\square_{r_{1},r_{2}}f\|_{\overline{p}}^{q_{2}}\frac{\dd h_{2}}{|h_{2}|}\Big)^{q_{1}/q_{2}}\frac{\dd h_{1}}{|h_{1}|}\Big)^{1/q_{1}}
  \end{align*}
  and
  \begin{align*}
    \|f\|_{\overline{\alpha},\overline{p},\overline{q}}^{(2)}  := & \, \|f\|_{\overline{p}} \, +  \Big( \int_{\R} h_{1}^{-\alpha_{1}q_{1}} \|\delta_{h_{1}}^1 f\|_{{\overline{p}}}^{q_{1}}\frac{\dd h_{1}}{|h_{1}|}  \Big)^{\frac{1}{q_1}}  + \Big( \int_{\T} h_2^{-\alpha_{2}q_{2}}  \|\delta_{h_{2}}^2 f\|_{\overline{p}}^{q_{2}}\frac{\dd h_{2}}{|h_{2}|}\Big)^{\frac{1}{q_2}}  \\
    &  + \,  \Big(\int_{\R}|h_{1}|^{-\alpha_{1}q_{1}}\Big(\int_{\T}|h_{2}|^{-\alpha_{2}q_{2}}\|\square_{h_{1},h_{2}}f\|_{\overline{p}}^{q_{2}}\frac{\dd h_{2}}{|h_{2}|}\Big)^{q_{1}/q_{2}}\frac{\dd h_{1}}{|h_{1}|}\Big)^{1/q_{1}}
  \end{align*}
  for $f\in\mathcal{S}_{\pi_2}'(\R^2)$ with the usual modification if $q_{1},q_{2},p_{1}$ or $p_{2}$ is equal to infinity.
\end{definition}

It turns out that these norms are equivalent Besov norms with dominating mixed smoothness.

\begin{theorem}\label{thm:characteristation}
  Let $\overline{\alpha}=(\alpha_{1},\alpha_{2})\in (0,1)^2$, $\overline{p} = (p_{1},p_{2}) \in[1,\infty]^2$ and $\overline{q} =(q_1,q_2) \in [1, \infty]^2$. Then, $\|\cdot\|_{\overline{\alpha},\overline{p},\overline{q}}^{(1)}$, $\|\cdot\|_{\overline{\alpha},\overline{p},\overline{q}}^{(2)}$ and $\|\cdot\|_{\overline{\alpha},\overline{p},\overline{q}}$ are equivalent norms on the space $B_{\overline{p},\overline{q}}^{\overline{\alpha}}(\R \times \T)$, that is,
  \begin{equation*}
    \|f\|_{\overline{\alpha},\overline{p},\overline{q}}\lesssim\|f\|_{\overline{\alpha},\overline{p},\overline{q}}^{(2)}\lesssim\|f\|_{\overline{\alpha},\overline{p},\overline{q}}^{(1)}\lesssim\|f\|_{\overline{\alpha},\overline{p},\overline{q}},\quad\text{for }f\in\mathcal{S}^{\prime}_{\pi_2}(\R^{2}) \cap L^{\overline{p}}(\R\times\mathbb T).
  \end{equation*}
  Furthermore, one can replace the integral domains in definitions of $\|\cdot\|_{\overline{\alpha},\overline{p},\overline{q}}^{(1)}$ and $\|\cdot\|_{\overline{\alpha},\overline{p},\overline{q}}^{(2)}$ by $h_{1},h_2\in[-1,1]$.
\end{theorem}

\begin{remark}
  A straight forward modification of our proof yields an analogous characterization of (non-periodic) Besov spaces with dominating mixed smoothness on $\R^2$ by replacing the torus by the real line in the definitions of $\|f\|_{\bar \alpha,\bar p,\bar q}^{(1)}$ and $\|f\|_{\bar \alpha,\bar p,\bar q}^{(2)}$.
\end{remark}

\subsection{Proof of Theorem~\ref{thm:characteristation}}

Before we prove the Theorem~\ref{thm:characteristation}, we provide several auxiliary results.

\begin{lemma}\label{lem:aux1}
  For any $f\in\mathcal{S}_{\pi_2}'(\R^2)$, $j_{1},j_{2}\ge-1$ and some $a>0$ the maximal function is defined as 
  \begin{equation}\label{eq:fstar}
    f_{j_{1},j_{2}}^{*}(x):=\sup_{y\in\R^{2}}\frac{|\Delta_{j_{1},j_{2}}f(x-y)|}{1+(2^{ 2 j_{1}}|y_{1}|^2+2^{ 2j_{2}}|y_{2}|^2)^{a/2}},\qquad x\in\R^{2}.
  \end{equation}
  Then, we have for all $k_{1},k_{2}\ge-1$ and $\bar p\in[1,\infty]^2$ that
  \begin{align*}
    &\sup_{|h_{1}|\le 2^{-k_{1}},|h_{2}|\le2^{-k_{2}}}\|\square_{(h_{1},h_{2})}(\Delta_{j_{1},j_{2}}f)\|_{\bar p}
     \lesssim   \min(1,2^{(j_{1}-k_{1})})\min(1,2^{(j_{2}-k_{2})})\|f_{j_{1},j_{2}}^{*}\|_{\bar p}, \\
    &\sup_{|h_{1}|\le2^{-k_{1}}}\|\delta_{h_{1}}(\Delta_{j_{1},j_{2}}f)\|_{\bar p}
     \lesssim   \min(1,2^{(j_{1}-k_{1})}) \|f_{j_{1},j_{2}}^{*}\|_{\bar p}, \\
    &\sup_{|h_{2}|\le2^{-k_{2}}}\|\delta_{h_{2}}(\Delta_{j_{1},j_{2}}f)\|_{\bar p}
     \lesssim  \min(1,2^{(j_{2}-k_{2})})\|f_{j_{1},j_{2}}^{*}\|_{\bar p}.
  \end{align*}
\end{lemma}

\begin{remark}
  The function \(f_{j_{1},j_{2}}^{*}\) is periodic in the second component. Therefore, the $L^{\bar p}$-norms in Lemma~\ref{lem:aux1} are well-defined.
\end{remark}

\begin{proof}
  \emph{Step~1:} We prove in the case $j_{i}\le k_{i},i\in\{0,1\}$ that
  \begin{equation*}
    |\square_{(h_{1},h_{2})}(\Delta_{j_{1},j_{2}}f)(x)|\lesssim2^{(j_{1}-k_{1})+(j_{2}-k_{2})}f_{j_{1},j_{2}}^{*}(x),\qquad x\in\R^{2}.
  \end{equation*}
  For \(|h_i| \le 2^{-k_i}\) we estimate
  \begin{align}\label{eq:aux1}
    &|\square_{(h_{1},h_{2})}(\Delta_{j_{1},j_{2}}f)(x)|\nonumber  \\
    & \quad\le2^{-k_{2}}\Big|\int_{0}^{1}\big(\partial_{2}(\Delta_{j_{1},j_{2}}f)(x_{1}+h_{1},x_{2}+th_{2})-\partial_{2}(\Delta_{j_{1},j_{2}}f)(x_{1},x_{2}+th_{2})\big)\,\dd t\Big|\nonumber \\
    & \quad\le2^{-(k_{1}+k_{2})}\int_{0}^{1}\int_{0}^{1}|\partial_{1}\partial_{2} (\Delta_{j_{1},j_{2}}f)(x_{1}+t_{1}h_{1},x_{2}+t_{2}h_{2})|\,\dd t_{2}\dd t_{1}\nonumber \\
    & \quad\le2^{-(k_{1}+k_{2})}\sup_{|z_{i}|
    \le 2^{-k_{i}},i=1,2}|\partial_{1}\partial_{2}(\Delta_{j_{1},j_{2}}f)(x_{1}-z_{1},x_{2}-z_{2})|\nonumber \\
    & \quad\lesssim 2^{-(k_{1}+k_{2})}\sup_{|z_{i}|\le2^{-k_{i}},i=1,2}\frac{|\partial_{1}\partial_{2}(\Delta_{j_{1},j_{2}}f)(x_{1}-z_{1},x_{2}-z_{2})|}{1+(2^{2 j_{1}}|z_{1}|^2+2^{2j_{2}}|z_{2}|^2)^{a/2}}\nonumber \\
    & \quad\le2^{-(k_{1}+k_{2})}\sup_{z\in\R^{2}}\frac{|\partial_{1}\partial_{2} (\Delta_{j_{1},j_{2}}f)(x_{1}-z_{1},x_{2}-z_{2})|}{1+(2^{2 j_{1}}|z_{1}|^2+2^{2j_{2}}|z_{2}|^2)^{a/2} },
  \end{align}
  where we have used $j_{i}\le k_{i}$ in the fourth step. In order to estimate the supremum we proceed similar to Lemma~\ref{lem:bernstein inequalities}. Let $\psi$ be a smooth compactly supported function satisfying $\psi(\xi)=1$ on the support of $\rho$. Setting $\Psi:=\F^{-1}[\psi\otimes\psi] \in \mathcal{S}(\R^{2})$, we have
  \begin{align*}
    \Delta_{j_{1},j_{2}}f & =\F^{-1}\big[(\rho_{j_{1}}\otimes\rho_{j_{2}})\F (f)\big(\psi(2^{-j_{1}}\cdot)\otimes\psi(2^{-j_{2}}\cdot)\big)\big]\\
    & =\Delta_{j_{1},j_{2}}f\ast\F^{-1}[\psi(2^{-j_{1}}\cdot)\otimes\psi(2^{-j_{2}}\cdot)]\\
    & =\Delta_{j_{1},j_{2}}f\ast\Big(2^{(j_{1}+j_{2})}\Psi(2^{j_{1}}\cdot,2^{j_{2}}\cdot)\Big).
  \end{align*}
  Therefore,
  \begin{align*}
    & | \partial_{1}\partial_{2}(\Delta_{j_{1},j_{2}}f)(x_{1}-z_{1},x_{2}-z_{2})  |  \\
    & \,\, = \Big| 2^{2(j_{1}+j_{2})}\int_{\R^{2}}\Delta_{j_{1},j_{2}}f(y)\partial_{1}\partial_{2}\Psi\big(2^{j_{1}}(x_{1}-z_{1}-y_{1}),2^{j_{2}}(x_{2}-z_{2}-y_{2})\big)\,\dd y \Big | \\
    &  \,\, \lesssim 2^{2(j_{1}+j_{2})}\int_{\R^{2}}|\Delta_{j_{1},j_{2}}f(x_{1}-y_{1},x_{2}-y_{2})|(1+2^{ 2 j_{1}}|y_{1}-z_{1}|^2+2^{2j_{2}}|y_{2}-z_{2}|^2)^{-b}\,\dd y\\
    &  \,\, \le2^{2(j_{1}+j_{2})}f_{j_{1},j_{2}}^{*}(x)\int_{\R^{2}}\big(1+(2^{ 2 j_{1}}|y_{1}|^2+2^{2 j_{2}}|y_{2}|^2)^{\frac{a}{2}}\big)(1+2^{2 j_{1}}|y_{1}-z_{1}|^2+2^{2j_{2}}|y_{2}-z_{2}|^2)^{-b}\,\dd y
  \end{align*}
  with some arbitrary large $b>1$. Next, applying the triangle inequality, we find
  \begin{equation*}
    1+(|2^{j_{1}}y_{1}|^2+|2^{ j_{2}}y_{2}|^2)^{a/2}
    \lesssim (1+(|2^{j_{1}}z_{1}|^2+|2^{j_{2}}z_{2}|^2)^{a/2} ) (1+2^{ 2 j_{1}}|y_{1}-z_{1}|^2+2^{2j_{2}}|y_{2}-z_{2}|^2)^{a/2},
  \end{equation*}
  and therefore we can estimate the supremum in \eqref{eq:aux1} as follows
  \begin{align*}
    &\sup_{z\in\R^{2}}\frac{\partial_{1}\partial_{2} (\Delta_{j_{1},j_{2}}f) (x_{1}-z_{1},x_{2}-z_{2})}{1+(|2^{j_{1}}z_{1}|^2+|2^{j_{2}}z_{2}|^2)^{a/2}}\\
    &\quad \lesssim  2^{2(j_{1}+j_{2})}f_{j_{1},j_{2}}^{*}(x) \sup_{z\in\R^{2}}  \int_{\R^{2}}(1+2^{j_{1}}|y_{1}-z_{1}|+2^{j_{2}}|y_{2}-z_{2}|)^{-c}\,\dd y \\
    &\quad \lesssim 2^{j_{1}+j_{2}}f_{j_{1},j_{2}}^{*}(x)\int_{\R^{2}}(1+|y_{1}|+|y_{2}|)^{-c}\,\dd y
  \end{align*}
  for $c=b-a/2$ which can be arbitrary large. Since the integral is finite, Step~1 is completed by combining the last line with \eqref{eq:aux1}.

  \emph{Step~2}: If $j_{1}> k_{1}$ and $j_{2}\le k_{2}$ we show that
  \begin{equation*}
    |\square_{(h_{1},h_{2})}(\Delta_{j_{1},j_{2}}f)(x)|\lesssim2^{j_{2}-k_{2}}\big(f_{j_{1},j_{2}}^{*}(x_{1}+h_{1},x_{2})+f_{j_{1},j_{2}}^{*}(x_{1},x_{2})\big).
  \end{equation*}
  Applying the same arguments in second coordinate only, we get for arbitrary \(b > 1\)
  \begin{align*}
    & |\square_{(h_{1},h_{2})}(\Delta_{j_{1},j_{2}}f)(x)|\\
    & \quad\le|(\Delta_{j_{1},j_{2}}f)(x_{1}+h_{1},x_{2}+h_{2})-(\Delta_{j_{1},j_{2}}f)(x_{1}+h_{1},x_{2})|\\
    & \quad\qquad+|(\Delta_{j_{1},j_{2}}f)(x_{1},x_{2}+h_{2})-(\Delta_{j_{1},j_{2}}f)(x_{1},x_{2})|\\
    & \quad\lesssim 2^{-k_{2}}\sup_{z_{2}\in\R}\frac{|\partial_{2}(\Delta_{j_{1},j_{2}}f)(x_{1}+h_{1},x_{2}-z_{2})|}{1+(2^{ 2 j_{2}}|z_{2}|^2)^{a/2}} + 2^{-k_{2}}\sup_{z_{2}\in\R}\frac{|\partial_{2}(\Delta_{j_{1},j_{2}}f)(x_{1},x_{2}-z_{2})|}{1+ ( 2^{2j_{2}}|z_{2}|^2)^{a/2}}\\
    & \quad\lesssim  2^{j_{1}+2j_{2}-k_{2}}\big(f_{j_{1},j_{2}}^{*}(x_{1}+h_{1},x_{2})+f_{j_{1},j_{2}}^{*}(x_{1},x_{2})\big)\\
    & \quad\qquad\qquad \sup_{z_{2}\in\R}\int_{\R^{2}}\frac{1+(2^{2   j_{1}}|y_{1}|^2+2^{2j_{2}}|y_{2}|^2)^{a/2}}{1+(2^{2 j_{2}}|z_{2}|^2)^{a/2}}(1+2^{2j_{1}}|y_{1}|^2+2^{2j_{2}}|y_{2}-z_{2}|^2)^{-b/2}\,\dd y\\
    & \quad\lesssim2^{j_{1}+2j_{2}-k_{2}}\big(f_{j_{1},j_{2}}^{*}(x_{1}+h_{1},x_{2})+f_{j_{1},j_{2}}^{*}(x_{1},x_{2})\big)\\
    & \quad\qquad\qquad\sup_{z_{2}\in\R}\int_{\R^{2}}\big(1+2^{ 2 j_{1}}|y_{1}|^2+2^{2 j_{2}}|y_{2}-z_{2}|^2 \big)^{a/2} \big(1+2^{2j_{1}}|y_{1}|^2+2^{2j_{2}}|y_{2}-z_{2}|^2 \big)^{-b/2}\,\dd y\\
    & \quad\lesssim2^{j_{2}-k_{2}}\big(f_{j_{1},j_{2}}^{*}(x_{1}+h_{1},x_{2})+f_{j_{1},j_{2}}^{*}(x_{1},x_{2})\big)\int_{\R^{2}}(1+|y_{1}|+|y_{2}|)^{-c}\,\dd y.
  \end{align*}
  The analogue statement holds true for the case \(j_1 \le k_1\), \(j_2 > k_2 \).

  \emph{Step~3:} We conclude the first inequality of the lemma.

  First note, that in the case \(j_i > k_i\) for \(i \in \{1,2\}\) we can use the translation invariance of the Lebesgue measure on \(\R \times \T\) such that we obtain
  \begin{align*}
    \|f_{j_{1},j_{2}}^{*}\|_{\bar p} & \ge\sup_{y\in\R^{2}}\frac{\|\Delta_{j_{1},j_{2}}f(\cdot -y)\|_{\bar p}}{1+(2^{j_{1}}|y_{1}|+2^{j_{2}}|y_{2}|)^{a}}\\
    & =\sup_{y\in\R^{2}}\frac{\|\Delta_{j_{1},j_{2}}f\|_{\bar p}}{1+(2^{j_{1}}|y_{1}|+2^{j_{2}}|y_{2}|)^{a}}=\|\Delta_{j_{1},j_{2}}f\|_{\bar p},
  \end{align*}
  which proves the assertion using the triangle inequality. Next, the case $j_{i}\le k_{i}$ for $i\in\{1,2\}$ follows from Step 1 while the case $j_{1}> k_{1}$ and $j_{2}\le k_{2}$ follows from Step 2 by taking $L^{\bar p}$-norm on both sides. Similarly, for $j_{1}\le k_{1}$ and $j_{2}> k_{2}$. 
  
  \emph{Step 4:} We conclude the second and third inequality.
  
  For \(|h_1| \le 2^{-k_1}\) and \(j_1 \le k_1\) we have
  \begin{align*}
    |\delta_{h_1}|
    &\le 2^{-k_1} \int\limits_{0}^1 \partial_1  (\Delta_{j_{1},j_{2}}f)(x_1+h_1t,x_2) \dd t \\
    &\le 2^{-k_1} \sup\limits_{|z_1| \le 2^{-k_1}}  \partial_1  (\Delta_{j_{1},j_{2}}f)(x_1-z_1,x_2).
  \end{align*}
  But at this moment we can repeat the arguments from Step 1 and 2 to obtain
  \begin{equation*}
    \sup_{|h_{1}|\le2^{-k_{1}}}\|\delta_{h_{1}}(\Delta_{j_{1},j_{2}}f)\|_{\bar p}
    \lesssim  2^{(j_{1}-k_{1})} \|f_{j_{1},j_{2}}^{*}\|_{\bar p}.
  \end{equation*}
  Moreover, for \(j_1 > k_1 \) the triangle inequality (see Step 3) implies
  \begin{equation*}
    \sup_{|h_{1}|\le2^{-k_{1}}}\|\delta_{h_{1}}(\Delta_{j_{1},j_{2}}f)\|_{\bar p}
    \lesssim  \|f_{j_{1},j_{2}}^{*}\|_{\bar p}.
  \end{equation*}
  Consequently, applying similar arguments for \(\delta_{h_2} f\) we conclude the lemma.
\end{proof}

\begin{lemma}\label{lem: maximal_type_inequality}
  Let \(\Omega \subseteq \R^2\) be a compact set, $\bar p=(p_1,p_2)\in[1,\infty]^2$ and \(r\in(0,1)\). Assume further that \(g \in L^{\overline{p}}(\R \times \T) \cap \mathcal{S}'_{\pi_2}(\R^2) \) satisfies \(\supp\mathcal{F} g \subseteq \Omega \). Then, the following inequality holds
  \begin{equation*}
    \Big\|  \sup\limits_{y \in \R^2} \frac{g(\cdot -y ) }{1+|y|^{2/r}}\Big \|_{\overline{p}}
    \lesssim \|g \|_{\overline{p}}.
  \end{equation*}
\end{lemma}

\begin{remark}
  The function \(g\) should be interpreted as a periodic function in the second component defined on \(\R^2\). This way the assumption \(\supp\mathcal{F} g \subseteq \Omega \) is reasonable in the sense of distributions.
\end{remark}

\begin{proof}
  Let \(f \in L^{\overline{p}}(\R\times\mathbb T)\) and \(x =(x_1,x_2)\). We define the iterated maximal function
  \begin{equation*}
    M(f)(x) := \sup\limits_{I_1 \in \mathcal{I}_1} \frac{1}{|I_1|}  \int_{I_1} \sup\limits_{I_2 \in \mathcal{I}_2}  \frac{1}{|I_1|}  \int_{I_2} |f(z_1,z_2)| \dd z_1 \dd z_2,
  \end{equation*}
  where \(\mathcal{I}_k\) denotes the set of all intervals containing \(x_k\).

  \emph{Step~1:} We first show that the inequality
  \begin{equation}\label{eq: maximal_inequality}
    \sup\limits_{y \in \R^2}  \frac{|\varphi(x -y )| }{1+|y|^{2/r}} \lesssim (M(|\varphi|^r) (x) )^{\frac{1}{r}}
  \end{equation}
  holds for \(\varphi \in \mathcal{S}(\R^2)\) with \(\supp\mathcal{F} \varphi \subseteq \Omega \) and all \(x \in \R^2\). Denote the cube around the origine with side length $l\in(0,1]$ by \(Q_l:= [-l,l] \times [-l,l]\). For any continuous differentiable function~\(h\) which is defined on \(Q_1\) the the mean value theorem yields
  \begin{equation*}
    |h(a)| \lesssim \inf\limits_{z \in Q_1} | h(z)| + \sup\limits_{z \in Q_1} | \nabla h(z)|
    \lesssim \Big( \int_{Q_1} |h(z)|^r \dd z \Big)^{\frac{1}{r}} + \sup\limits_{z \in Q_1} | \nabla h(z)| , \quad  a \in Q_1.
  \end{equation*}
  A simple scaling argument shows that for all functions \(h\), which are continuously differentiable on \(Q_l\), we have
  \begin{equation*}
    |h(0)|=|h(l\cdot 0)| \lesssim  l^{-\frac{2}{r}} \Big( \int_{Q_l} |h(z)|^r \dd z \Big)^{\frac{1}{r}} + l \sup\limits_{z \in Q_l} | \nabla h(z)|.
  \end{equation*}

  Next, let us fix \(x,y \in \R^2\), \(l \in (0,1)\) and set \(h(a) := \varphi(x-y-a)\). We also denote by \(Q_l^{x-y}\) the cube with side length \(l\) and center \(x-y\). Then, using the previous inequality, we find
  \begin{align*}
    |\varphi(x-y)| = |h(0)|
    & \lesssim  l^{-\frac{2}{r}} \Big( \int_{Q_l} |h(z)|^r \dd z \Big)^{\frac{1}{r}} + l \sup\limits_{z \in Q_l} | \nabla h(z)| \\
    &= l^{-\frac{2}{r}} \Big( \int_{Q_l} |\varphi(x-y-z)|^r \dd z \Big)^{\frac{1}{r}} + l \sup\limits_{z \in Q_l} | \nabla \varphi(x-y-z)| .
  \end{align*}
  Now, the integral on the right-hand side can be bounded by
  \begin{align*}
    \Big( \int_{Q_1^{x-y}} |\varphi(z)|^r \dd z \Big)^{\frac{1}{r}}
    &= \Big( \int_{[x_1-y_1-1,x_1-y_1+1]} \int_{[x_2-y_2-1,x_2-y_2+1]}  |\varphi(z)|^r \dd z_2 \dd z_2  \Big)^{\frac{1}{r}} \\
    &\le \Big( \int_{[x_1-|y_1|-1,x_1+|y_1|+1]} \int_{[x_2-|y_2|-1,x_2+|y_2|+1]}  |\varphi(z)|^r \dd z_2 \dd z_2  \Big)^{\frac{1}{r}}  \\
    &\le (4(1+|y_1|)(1+|y_2|))^{\frac{1}{r}} (M(|\varphi|^r)(x) )^{\frac{1}{r}}   \\
    &\lesssim (1+|y|^{2/r}) (M(|\varphi|^r)(x) )^{\frac{1}{r}} ,
  \end{align*}
  where in the last step we used the inequality
  \begin{equation*}
    (1+|y_1|)^{\frac{1}{r}}(1+|y_2|)^{\frac{1}{r}}
    \leq (1+|y|)^{\frac{2}{r}}
    \lesssim (1+|y|^{\frac{2}{r}}) .
  \end{equation*}
  As a result, dividing by \((1+|y|^{2/r})\) and then taking the supremum over \(y \in \R^2\) we obtain
  \begin{equation*}
    \sup\limits_{y \in \R^2} \frac{|\varphi(x-y)|}{1+|y|^{2/r}} \lesssim  l^{-\frac{2}{r}}  (M(|\varphi|^r)(x) )^{\frac{1}{r}}
    + l \sup\limits_{y \in \R^2}   \frac{| \nabla \varphi(x-y)| }{1+|y|^{2/r}}.
  \end{equation*}
  The assumptions on $\phi$ allow us to insert \cite[Theorem~1.3.1]{Triebel2010}, i.e.
  \begin{equation*}
    \sup\limits_{y \in \R^2}   \frac{| \nabla \varphi(x-y)| }{1+|y|^{2/r}}
    \lesssim \sup\limits_{y \in \R^2} \frac{|\varphi(x-y)|}{1+|y|^{2/r}}
  \end{equation*}
  in the previous inequality in order to obtain
  \begin{equation*}
    \sup\limits_{y \in \R^2} \frac{|\varphi(x-y)|}{1+|y|^{2/r}} \lesssim  l^{-\frac{2}{r}}  (M(|\varphi|^r)(x) )^{\frac{1}{r}}
    + l \sup\limits_{y \in \R^2} \frac{|\varphi(x-y)|}{1+|y|^{2/r}}.
  \end{equation*}
  Choosing \(l>0\) small enough and absorbing the term on the right-hand side proves \eqref{eq: maximal_inequality}.

  \emph{Step~2:} In this step we want to prove the claim of the lemma. For \(p=\infty\) the claim is obvious. For \(1 \le p_i < \infty\) let \(\psi \in \mathcal{S}(\R^2)\) with \(\psi(0)=1 \) and \(\supp\mathcal{F} \psi \subseteq B_1(0) \). Define \(g_{\tau}(x) := \psi(\tau x) g(x) \in \mathcal{S}(\R^2)\) and observe that the function $g_{\tau}(x)$ is in \(\mathcal{S}(\R^2)\) by the Paley--Wiener--Schwartz theorem. Then, since the Fourier transformation transforms products of two functions into convolutions we have \(\supp \mathcal{F} g_{\tau} \subseteq B_{\mathrm{diam}(\Omega)+1}(0)\). Consequently, \eqref{eq: maximal_inequality} implies
  \begin{equation} \label{eq: pointwise_max_inequality}
    \sup\limits_{y \in \R^2}  \frac{|g_{\tau}(x -y )| }{1+|y|^{2/r}}
    \lesssim (M(|g_{\tau}|^r) (x) )^{\frac{1}{r}}
    \lesssim  (M(|g|^r) (x) )^{\frac{1}{r}}.
  \end{equation}
  Moreover, \(g_\tau\) converges in \(L^\infty\) to \(g\) as \(\tau \to 0\). Indeed, let \(\tilde{\varphi}\) be a Schwartz function such that \( \mathcal{F} \tilde{\varphi} = 1 \) on \( B_{\mathrm{diam}(\Omega)+1}(0)\). Then, we have
  \begin{equation*}
    g_\tau(x)- g_{\tau'}(x) = \mathcal{F}^{-1}(\mathcal{F}  g_\tau \mathcal{F} \tilde{\varphi} -\mathcal{F}  g_{\tau'} \mathcal{F} \tilde{\varphi})(x)
    = \mathcal{F}^{-1}(\mathcal{F}  (g_\tau* \tilde{\varphi} ) -\mathcal{F}  (g_{\tau'} * \tilde{\varphi}))(x) =  (g_\tau- g_{\tau'})*\tilde{\varphi}(x).
  \end{equation*}
  By Young's inequality and dominated convergence, this implies that the family \((g_\tau, \tau > 0)\) is a Cauchy sequence in \(L^\infty\). Hence, it converges to some \(\tilde{g} \in L^\infty\). By the fundamental lemma of calculus of variations it is straightforward to show that \(\tilde{g} = g\) almost everywhere.
   Hence, after taking \(\tau \to 0\) in~\eqref{eq: pointwise_max_inequality}, we integrate both sides to find
  \begin{equation*}
    \Big \| \sup\limits_{y \in \R^2} \frac{g(\cdot -y ) }{1+|y|^{2/r}} \Big \|_{\overline{p}}
    \lesssim \| (M(|g|^r) )^{\frac{1}{r}} \|_{\overline{p}}.
  \end{equation*}
  Since \(p_i/r >1 \) we can apply \cite[Theorem~2.16]{Huang2021} (after reducing it to the \(\R^2\) case in the style of~\cite[Proposition 3.2.4.]{Schmeisser1987}), which states that the iterated maximal function is a bounded operator from \(L^{\frac{\overline{p}}{r}} \to L^{\frac{\overline{p}}{r}}\), and we obtain
  \begin{equation*}
    \Big \|  \sup\limits_{y \in \R^2} \frac{g(\cdot -y ) }{1+|y|^{2/r}} \Big  \|_{\overline{p}}
    \lesssim  \| g \|_{\overline{p}},
  \end{equation*}
  which proves the lemma.
\end{proof}

\begin{corollary}\label{cor:aux2}
  For $a>2$ and $\bar p\in[1,\infty]^2$ the function $f^*_{j_1,j_2}$ from \eqref{eq:fstar} satisfies
  \begin{equation*}
    \|f_{j_{1},j_{2}}^{*}\|_{\bar p}\le c\,\|\Delta_{j_{1},j_{2}}f\|_{\bar p},
  \end{equation*}
  for a constant $c>0$ that does not depend on $f,j_{1},j_{2}$.
\end{corollary}

\begin{proof}
  Rescaling $\Delta_{j_{1},j_{2}}f$, we define
  \begin{equation*}
    g_{j_{1},j_{2}}(x_{1},x_{2})  :=\Delta_{j_{1},j_{2}}f(2^{-j_{1}}x_{1},2^{-j_{2}}x_{2})
  \end{equation*}
  such that the Fourier transform of $g_{j_{1},j_{2}}$ is supported in $[-C,C]^{2}$ for some fixed $C$ depending only on $\rho$. Then, we calculate
  \begin{align*}
     \|f_{j_{1},j_{2}}^{*}\|_{\bar p}
    &= \Big(\int_{\R}\Big(\int_{\T}\sup_{y\in\R^{2}}\frac{|\Delta_{j_{1},j_{2}}f(x-y)|^{p_{2}}}{(1+(2^{2j_{1}}|y_{1}|^2+2^{2 j_{2}}|y_{2}|^2)^{a/2})^{p_{2}}}\dd x_{2}\Big)^{p_{1}/p_{2}}\dd x_{1}\Big)^{1/p_{1}}\\
    &\lesssim\Big(\int_{\R}\Big(\int_{\T}\sup_{y \in\R^2}\frac{|g_{j_{1},j_{2}}(2^{j_{1}}x_{1}-2^{j_{1}}y_{1},2^{j_{2}}x_{2}-2^{j_{2}}y_{2})|^{p_{2}}}{(1+(2^{2j_{1}}|y_{1}|^2+2^{2 j_{2}}|y_{2}|^2)^{a/2})^{p_{2}}}\dd x_{2}\Big)^{p_{1}/p_{2}}\dd x_{1}\Big)^{1/p_{1}}\\
    &  = \Big(\int_{\R}\Big(\int_{\T}\sup_{y \in\R^2}\frac{|g_{j_{1},j_{2}}(2^{j_{1}}x_{1}-y_{1},2^{j_{2}}x_{2}-y_{2})|^{p_{2}}}{(1+|y_{1}|^2+|y_{2}|^2)^{a/2})^{p_{2}}}\dd x_{2}\Big)^{p_{1}/p_{2}}\dd x_{1}\Big)^{1/p_{1}}\\
    &  \le \Big(\int_{\R}\Big(\int_{\T}\sup_{y \in\R^2}\frac{|g_{j_{1},j_{2}}(2^{j_{1}}x_{1}-y_{1},2^{j_{2}}x_{2}-y_{2})|^{p_{2}}}{(1+|y|^{a})^{p_{2}}}\dd x_{2}\Big)^{p_{1}/p_{2}}\dd x_{1}\Big)^{1/p_{1}}\\
       &  \le \Big(\int_{\R}\Big(\int_{\T} M(|g|^r)^{\frac{p_2}{r}} (2^{j_1}x_1,2^{j_2}x_2) \dd x_{2}\Big)^{p_{1}/p_{2}}\dd x_{1}\Big)^{1/p_{1}}\\
    &=  \Big(\int_{\R}\Big(\int_{\T} M(|\Delta_{j_{1},j_{2}}f|^r)^{\frac{p_2}{r}}(x_{1},x_{2}) \dd x_{2}\Big)^{p_{1}/p_{2}}\dd x_{1}\Big)^{1/p_{1}}\\
    & \lesssim \|\Delta_{j_{1},j_{2}}f\|_{\bar p},
  \end{align*}
  where we used inequality~\eqref{eq: pointwise_max_inequality} with \(r= \frac{2}{a}\), \(r \le \min(p_1,p_2)\) in the fifth step and the strong continuity of the maximal operator similar to Lemma~\ref{lem: maximal_type_inequality}.

\end{proof}

\begin{lemma}\label{lem:aux3}
  Let $\bar{p}\in[1,\infty]^{2}$. For every $f\in L^{\bar p}(\R \times \T)$ there exists a sequence \((g_{j_1,j_2})_{j_1,j_2\ge 1 } \subset L^{\overline{p}}(\R \times \T )\) with $\supp\F g_{j_{1},j_{2}}\subseteq\{\xi\in\R^{2}:|\xi_{1}|\le2^{j_{1}} \; \textrm{and} \; |\xi_{2}|\le2^{j_{2}}\}$ such that \( \lim_{j_1 \wedge j_2 \to \infty} g_{j_1,j_2} = f \) in \(\mathcal{S}^{\prime}(\R^{2})\). Furthermore, there is a constant \(c\) such that
  \begin{align}
    \label{eq: square_l_p_estimate}
    &\|g_{j_{1},j_{2}}^{\square}\|_{\bar{p}} \le c \sup\limits_{|r_i|\le 2^{-j_{i}}} \|\square_{r}f\|_{\bar{p}} ,  \\
    \label{eq: delta_1_l_p_estimate}
    &\|g_{j_1,1}-g_{j_1-1,1}\| _{\bar{p}} \le c \sup\limits_{|r_1|\le 2^{-j_{1}}} \|\delta_{r_1}f\|_{\bar{p}} , \\
    \label{eq: delta_2_l_p_estimate}
    &\|g_{1,j_2}-g_{1,j_2-1}\| _{\bar{p}} \le c \sup\limits_{|r_2|\le 2^{-j_{2}}} \|\delta_{r_2}f\|_{\bar{p}},
  \end{align}
  for \(j_1,j_2 \ge 2 \), and
  \begin{equation*}
    g_{j_{1},j_{2}}^{\square}:=g_{j_{1},j_{2}}-g_{j_{1}-1,j_{2}}-g_{j_{1},j_{2}-1}+g_{j_{1}-1,j_{2}-1}.
  \end{equation*}
\end{lemma} 

\begin{proof}
  \emph{Step 1:} We show a first estimate of $\|g_{j_{1},j_{2}}^{\square}\|_{\bar{p}}$ in terms of $\|\square_{r}f\|_{\bar{p}}$.
  To this end, we generalize the approach by \cite[Chapter~5.2.1]{Nikolskji1975} to the rectangular increments. Let $K\colon\R\to\R$ be a Schwartz function with smooth Fourier transform supported in $[-1,1]$ and $\int K_{b}(x)\dd x=1$. Set $K_{a}:=a^{-1}K(a^{-1}\cdot)$ for $a>0$. Then, for \(b_1,b_2\in(0,1)\), the function
  \begin{equation*}
    g_{b_1,b_2}(x) :=  (K_{b_{1}}\otimes K_{b_{2}})\ast f
  \end{equation*}
  satisfies $\supp\mathcal{F}g_{b_1,b_2}\subseteq\{\xi\in\R^{2}:|\xi_{1}|\le b_{1}^{-1} \; \textrm{and} \; |\xi_{2}|\le b_{2}^{-1}\}$ as well as the convergence to $f$ in \(\mathcal{S}^{\prime}(\R^{2})\) for $b_1\vee b_2\to 0$. Furthermore, the sequence is periodic in the second component, which follows by   a simple calculation. Choosing \(b= (2^{-j_1},2^{-j_2}) \), we obtain the sequence \((g_{j_1,j_2})_{j_1,j_2\ge 1 } \) and the first part of the lemma. Now, for the second part we notice that for \(h=(h_1,h_2)\) we have
  \begin{equation*}
    K_{2^{-j_i}}(h_{i})
    = 2^{j_i}  K(2^{j_1} h_{i})
    = 2  2^{j_i-1}  K( 2 2^{j_1-1} h_{i})
    = 2 K_{2^{-(j_i-1)}}(2 h_{i})
  \end{equation*}
  for \(i=1,2\). Hence, we can compute the following 
  \begin{align*}
    (K_{2^{-(j_1-1)}} \otimes K_{2^{-j_2}})\ast f
    &= \int_{\R^2} K_{2^{-(j_1-1)}}(h_1) K_{2^{-j_2}}(h_2) f(x_1-h_1,x_2-h_2) \dd h_1 \dd h_2 \\
    &=  \int_{\R^2} 2 K_{2^{-(j_1-1)}}( 2h_1) K_{2^{-j_2}}(h_2) f(x_1-2h_1,x_2-h_2) \dd h_1 \dd h_2 \\
    &=  \int_{\R^2} K_{2^{-j_i}}(h_{1}) K_{2^{-j_2}}(h_2) f(x_1-2h_1,x_2-h_2) \dd h_1 \dd h_2  .
  \end{align*}
  In a similar manner we obtain
  \begin{equation*}
    (K_{2^{-j_1}} \otimes K_{2^{-(j_2-1)}})\ast f
    =  \int_{\R^2} K_{2^{-j_i}}(h_{1}) K_{2^{-j_2}}(h_2) f(x_1-h_1,x_2-2h_2) \dd h_1 \dd h_2
  \end{equation*}
  and 
  \begin{equation*}
    (K_{2^{-(j_1-1)}} \otimes K_{2^{-(j_2-1)}})\ast f
    =  \int_{\R^2} K_{2^{-j_i}}(h_{1}) K_{2^{-j_2}}(h_2) f(x_1-2h_1,x_2-2h_2) \dd h_1 \dd h_2.
  \end{equation*}
  Consequently, we find 
  \begin{align*}
    g_{j_{1},j_{2}}^{\square}
    =& \;  (K_{2^{-j_1}} \otimes K_{2^{-j_2}})\ast f -(K_{2^{-(j_1-1)}} \otimes K_{2^{-j_2}})\ast f  -  (K_{2^{-j_1}} \otimes K_{2^{-(j_2-1)}})\ast f \\
    &\qquad+ (K_{2^{-(j_1-1)}} \otimes K_{2^{-(j_2-1)}})\ast f \\
    = & \;  \int_{\R^2} K_{2^{-j_i}}(h_{1}) K_{2^{-j_2}}(h_2) ( f(x_1-h_1,x_2-h_2) \\
    &\qquad- f(x_1-2h_1,x_2-h_2) - f(x_1-h_1,x_2-2h_2)+ f(x_1-2h_1,x_2-2h_2)) \dd h_1 \dd h_2   \\
    = & \;  \int_{\R^2} K_{2^{-j_i}}(-h_{1}) K_{2^{-j_2}}(-h_2) ( f(x_1+h_1,x_2+h_2) \\
    &\qquad- f(x_1+2h_1,x_2+h_2) - f(x_1+h_1,x_2+2h_2)  + f(x_1+2h_1,x_2+2h_2)) \dd h_1 \dd h_2  \\
    =& \;  \int_{\R^2} K_{2^{-j_i}}(-h_{1}) K_{2^{-j_2}}(-h_2) \square_{h}f(x_{1}+h_1,x_{2}+h_2) \dd h_1 \dd h_2.
  \end{align*}
  Next, Minkowski's integral inequality yields
  \begin{align*}
    \|g_{j_{1},j_{2}}^{\square}\|_{\bar{p}} & =\Big(\int_{\R}\Big(\int_{\T}\Big|\int_{\R^{2}}(K_{2^{-j_1}} \otimes K_{2^{-j_2}})(-h) \square_{h}f(x_{1}+h_1,x_{2}+h_2) \dd h\Big|^{p_{2}}\dd x_{2}\Big)^{p_{1}/p_{2}}\dd x_{1}\Big)^{\frac{1}{p_{1}}}\\
    & \le \Big(\int_{\R}\Big(\int_{\R^{2}}|(K_{b_{1}}\otimes K_{b_{2}})(-h)|\Big(\int_{\T}\big| \square_{h}f(x_{1}+h_1,x_{2}+h_2) \big|^{p_{2}}\dd x_{2}\Big)^{1/p_{2}}\dd h\Big)^{p_{1}}\dd x_{1}\Big)^{\frac{1}{p_{1}}}\\
    & \le \int_{\R^{2}}|(K_{2^{-j_1}} \otimes K_{2^{-j_2}})(-h) |\|\square_{h}f(\cdot+h_1,\cdot+h_2)\|_{\bar{p}}\dd h\\
    & = \int_{\R^{2}}|(K_{2^{-j_1}} \otimes K_{2^{-j_2}})(-h) |\|\square_{h}f \|_{\bar{p}}\dd h\\
    & \le\int_{\R^{2}}|(K_{2^{-j_1}} \otimes K_{2^{-j_2}})(-h) |\sup_{|r_{i}|\le|h_{i}|}\|\square_{r}f\|_{\bar{p}}\dd h\\
    & =\int_{\R^{2}}|(K\otimes K)(-h)|\sup_{|r_{i}|\le 2^{-j_i}|h_{i}|}\|\square_{r}f\|_{\bar{p}}\dd h.
  \end{align*}
  
  \emph{Step 2:}  We will prove \eqref{eq: square_l_p_estimate}. In order to estimate $\sup_{|r_{i}|\le 2^{-j_i}|h_{i}|}\|\square_{r}f\|_{\bar{p}}$ in the previous line, we need to establish the following property
  \begin{equation}\label{eq: square_scaling}
    \sup\limits_{|r_i| \le h_{i} l_{i}} \|\square_{(r_1,r_2)} \tilde{f}\|_{\bar{p}}
    \le (1+h_1)(1+h_2)  \sup\limits_{|r_i| \le l_{i}} \|\square_{(r_1,r_2)} \tilde{f}\|_{\bar{p}}
  \end{equation}
  for \(h=(h_1,h_2)$, $l=(l_1,l_2) \in (0,\infty)^2 \) and $ \tilde{f} \in L^{\overline{p}}(\R^{2})$. 
  First, we note that, for integers $n_1, n_2 \in \N $ and \(|r_i| \le n_{i} l_{i}\), we obtain with \(\tilde{r}_{i} := r_{i}/n_{i} \) for \(i=1,2\) that
  \begin{align*}
    \|\square_{(r_1,r_2)} \tilde{f}\|_{\bar{p}}
    = & \|\delta_{r_2} ( \tilde{f}(\cdot+r_1,\cdot)-\tilde{f}(\cdot,\cdot) ) \|_{\bar{p}} \\
    = & \|\delta_{r_2} ( \tilde{f}(\cdot+\tilde{r}_{1} n_{i},\cdot)-\tilde{f}(\cdot,\cdot) ) \|_{\bar{p}}  \\
    = & \Big\|\delta_{r_2} \Big( \sum\limits_{j=1}^{n_1} \tilde{f}(\cdot+j \tilde{r}_{1} ,\cdot)-\tilde{f}(\cdot+(j-1)\tilde{r}_{1},\cdot) \Big) \Big\|_{\bar{p}}  \\
    \le & \sum\limits_{j=1}^{n_1} \|\delta_{r_2} ( \tilde{f}(\cdot+j \tilde{r}_{1} ,\cdot)-\tilde{f}(\cdot+(j-1)\tilde{r}_{1},\cdot) ) \|_{\bar{p}}  \\
    =& n_1\| \delta_{r_2} ( \tilde{f}(\cdot+ \tilde{r}_{1} ,\cdot)-\tilde{f}(\cdot,\cdot) ) \|_{\bar{p}} \\
    =& n_1\| \delta_{\tilde r_1} ( \tilde{f}(\cdot,\cdot + n_2\tilde r_2)-\tilde{f}(\cdot,\cdot) ) \|_{\bar{p}} \\
    \le & n_1\sum\limits_{j=1}^{n_2} \|\delta_{\tilde r_1} ( \tilde{f}(\cdot,\cdot+j \tilde{r}_{2})-\tilde{f}(\cdot,\cdot+(j-1)\tilde{r}_{2}) \|_{\bar{p}}  \\
    = & n_1 n_2\Big \|  \square_{(\tilde r_1,\tilde r_2)}\tilde{f} \Big\|_{\bar{p}} .
  \end{align*}
  Hence, $ \sup\limits_{|r_i| \le n_{i} l_{i}} \|\square_{(r_1,r_2)} \tilde{f}\|_{\bar{p}} \le  n_1 n_2  \sup\limits_{|r_i| \le  l_{i}} \|\square_{(r_1,r_2)} \tilde{f}\|_{\bar{p}}$. Using this estimate we deduce \eqref{eq: square_scaling} as follows
  \begin{align*}
    \sup\limits_{|r_i| \le h_{i} l_{i}} \|\square_{(r_1,r_2)} \tilde{f}\|_{\bar{p}}
    &\le \sup\limits_{|r_i| \le \lceil h_{i} \rceil l_{i}} \|\square_{(r_1,r_2)} \tilde{f}\|_{\bar{p}} \\
    &\le  \lceil h_{1} \rceil  \lceil h_{2} \rceil
    \sup\limits_{|r_i| \le l_{i}} \|\square_{(r_1,r_2)} \tilde{f}\|_{\bar{p}} \\
    &\le (1+h_1)(1+h_2)  \sup\limits_{|r_i| \le l_{i}} \|\square_{(r_1,r_2)} \tilde{f}\|_{\bar{p}},
  \end{align*}
  where $\lceil z \rceil := \inf \{n\in \N : n\geq z\} $ for $z\in \R$. In combination with Step~1 we conclude
  \begin{align*}
    \|g_{j_{1},j_{2}}^{\square}\|_{\bar{p}} 
    & \le  \int_{\R^{2}}|(K\otimes K)(-h)|\sup_{|r_{i}|\le 2^{-j_i}|h_{i}|}\|\square_{r}f\|_{\bar{p}}\dd h \\
    & \le \int_{\R^{2}}|(K\otimes K)(-h)|(1+|h_{1}|)(1+|h_{2}|)\sup_{|r_{i}|\le 2^{-j_i}}\|\square_{r}f\|_{\bar{p}}\dd h\\
    & =\sup_{|r_{i}|\le 2^{-j_i}}\|\square_{r}f\|_{\bar{p}}\Big(\int_{\R}K(h)(1+|h|)\dd h\Big)^{2}.
  \end{align*}
  Since $K$ is rapidly decaying the last integral is finite and the estimate \eqref{eq: square_l_p_estimate} is done.

  \emph{Step 3:} We want to show \eqref{eq: delta_1_l_p_estimate} and \eqref{eq: delta_2_l_p_estimate}. We use the same representation as before, namely
  \begin{equation*}
    (K_{2^{-(j_1-1)}} \otimes K_{2^{-1}})\ast f
    =  \int_{\R^2} K_{2^{-j_i}}(h_{1}) K_{2^{-1}}(h_2) f(x_1-2h_1,x_2-h_2) \dd h_1 \dd h_2
  \end{equation*}
  to obtain 
  \begin{align*}
    g_{j_1,1}-g_{j_1-1,1}
    &= (K_{2^{-(j_1)}} \otimes K_{2^{-1}})\ast f - (K_{2^{-(j_1-1)}} \otimes K_{2^{-1}})\ast f  \\
    &= \int_{\R^2} K_{2^{-j_1}}(h_{1}) K_{2^{-1}}(h_2) (f(x_1-h_1,x_2-h_2) - f(x_1-2h_1,x_2-h_2) ) \dd h_1 \dd h_2  \\
    &=\int_{\R^2} K_{2^{-j_1}}(-h_{1}) K_{2^{-1}}(-h_2) \delta_{h_1}f(x_1+h_1,x_2+h_2)  \dd h_1 \dd h_2 .
  \end{align*}
  Again, applying Minkowski's integral inequality and a variable transformation in \(x_1 \) we find 
  \begin{align*}
    \|g_{j_1,1}-g_{j_1-1,1} \|_{\bar{p}}
    &\le \int_{\R^{2}}|(K_{2^{-j_1}} \otimes K_{2^{-1}})(-h) |  \|\delta_{h_1}f \|_{\bar{p}}\dd h\\
    &\lesssim \int_{\R} |(K_{2^{-j_1}}(-h_1)|  \|\delta_{h_1}f \|_{\bar{p}}\dd h_1\\
    &\le \int_{\R} |K(-h_1)|  \sup\limits_{|r_1| \le 2^{-j_1}|h_1|}  \|\delta_{r_1}f \|_{\bar{p}}\dd h_1 .
  \end{align*}
  Analogously to \eqref{eq: square_scaling} or using \cite[Chapter~4.2 eq.~(8)]{Nikolskji1975} we have
  \begin{equation*}
    \sup\limits_{|r_1| \le 2^{-j_1}|h_1|}  \|\delta_{r_1}f \|_{\bar{p}}
    \le (1+|h_1|) \sup\limits_{|r_1| \le 2^{-j_1}}\|\delta_{r_1}f \|_{\bar{p}}.
  \end{equation*}
  Hence, we obtain
  \begin{equation*}
    \|g_{j_1,1}-g_{j_1-1,1} \|_{\bar{p}}
    \le  \sup\limits_{|r_1| \le 2^{-j_1}}\|\delta_{r_1}f \|_{\bar{p}} \int_{\R} |K(-h_1)|(1+|h_1|)  \dd h_.
  \end{equation*}
  As before, the integral on the right hand side is finite since \(K\) is rapidly decreasing and therefore \eqref{eq: delta_1_l_p_estimate} is proved.
  In the same way we can verify \eqref{eq: delta_2_l_p_estimate}.
\end{proof}

With these auxiliary lemmata at hand we can prove over main theorem.

\begin{proof}[Proof of Theorem~\ref{thm:characteristation}]
  \emph{Step~1:} We show $\|f\|_{\bar{\alpha},\bar{p},\bar{q}}^{(1)}\lesssim\|f\|_{\bar{\alpha},\bar{p},\bar{q}}$ for all $f\in B_{\overline{p},\overline{q}}^{\overline{\alpha}}(\R \times \T)$.

  Let $f\in B_{\bar{p},\bar{q}}^{\bar{\alpha}}(\R \times \T)$ with anisotropic Littlewood--Paley decomposition $f=\sum_{j,k\ge-1}\Delta_{j,k}f.$ Due to Lemma~\ref{lem:embeddings}, we first note that $f\in L^{\overline{p}}(\R \times \T)$ and $\|f\|_{\bar{p}}\lesssim\|f\|_{\bar{p},\bar{q}}^{\bar{\alpha}}$.

  By the triangle inequality and Lemma~\ref{lem:aux1}, we can bound
  \begin{align*}
    \sup_{|r_{i}|\le2^{-k_{i}}}\|\square_{r_{1},r_{2}}f\|_{\bar p} & =\sup_{|r_{i}|\le2^{-k_{i}}}\Big(\int_{\R}\Big(\int_{\T}|\square_{(r_{1},r_{2})}f(x_{1},x_{2})|^{p_{2}}\dd x_{2}\Big)^{p_{1}/p_{2}}\d x_{1}\Big)^{1/p_{1}}\\
    & \le\sum_{j_{1},j_{2}\ge-1}\sup_{|r_{i}|\le2^{-k_{i}}}\|\square_{(r_{1},r_{2})}(\Delta_{j_{1},j_{2}}f)\|_{\bar p}\\
    & \lesssim\sum_{j_{1},j_{2}\ge-1}\min(1,2^{(j_{1}-k_{1})})\min(1,2^{(j_{2}-k_{2})})\|f_{j_{1},j_{2}}^{*}\|_{\bar p}.
  \end{align*}
  Discretizing the integrals, we have
  \begin{align*}
    & \Big(\int_{\R}|h_{1}|^{-\alpha_{1}q_{1}}\Big(\int_{\T}|h_{2}|^{-\alpha_{2}q_{2}}\sup_{|r_{1}|\le|h_{1}|,|r_{2}|\le h_{2}}\|\square_{r_{1},r_{2}}f\|_{\overline{p}}^{q_{2}}\frac{\dd h_{2}}{|h_{2}|}\Big)^{q_{1}/q_{2}}\frac{\dd h_{1}}{|h_{1}|}\Big)^{1/q_{1}} \\
    & \quad\lesssim \Big(\sum_{k_{1}=0}^{\infty}2^{\alpha_{1}q_{1}k_{1}}\Big(\sum_{k_{2}=0}^{\infty}2^{\alpha_{2}q_{2}k_{2}}\sup_{|r_{i}|\le2^{-k_{i}}}\|\square_{r_{1},r_{2}}f\|_{\bar p}^{q_{2}}\Big)^{q_{1}/q_{2}}\Big)^{1/q_{1}}\\
    & \quad\lesssim \Big(\sum_{k_{1}=0}^{\infty}2^{\alpha_{1}q_{1}k_{1}}\Big(\sum_{k_{2}=0}^{\infty}2^{\alpha_{2}q_{2}k_{2}}\Big(\sum_{j_{1}=-1}^{k_{1}}\sum_{j_{2}=-1}^{k_{2}}2^{(j_{1}-k_{1})+(j_{2}-k_{2})}\|f_{j_{1},j_{2}}^{*}\|_{\bar p}\Big)^{q_{2}}\Big)^{q_{1}/q_{2}}\Big)^{1/q_{1}}\\
    & \quad\qquad+\Big(\sum_{k_{1}=0}^{\infty}2^{\alpha_{1}q_{1}k_{1}}\Big(\sum_{k_{2}=0}^{\infty}2^{\alpha_{2}q_{2}k_{2}}\Big(\sum_{j_{1}=-1}^{k_{1}}\sum_{j_{2}>k_{2}}2^{(j_{1}-k_{1})}\|f_{j_{1},j_{2}}^{*}\|_{\bar p}\Big)^{q_{2}}\Big)^{q_{1}/q_{2}}\Big)^{1/q_{1}}\\
    & \quad\qquad+\Big(\sum_{k_{1}=0}^{\infty}2^{\alpha_{1}q_{1}k_{1}}\Big(\sum_{k_{2}=0}^{\infty}2^{\alpha_{2}q_{2}k_{2}}\Big(\sum_{j_{1}>k_{1}}\sum_{j_{2}=-1}^{k_{2}}2^{(j_{2}-k_{2})}\|f_{j_{1},j_{2}}^{*}\|_{\bar p}\Big)^{q_{2}}\Big)^{q_{1}/q_{2}}\Big)^{1/q_{1}}\\
    & \quad\qquad+\Big(\sum_{k_{1}=0}^{\infty}2^{\alpha_{1}q_{1}k_{1}}\Big(\sum_{k_{2}=0}^{\infty}2^{\alpha_{2}q_{2}k_{2}}\Big(\sum_{j_{1}>k_{1}}\sum_{j_{2}>k_{2}}\|f_{j_{1},j_{2}}^{*}\|_{\bar p}\Big)^{q_{2}}\Big)^{q_{1}/q_{2}}\Big)^{1/q_{1}}\\
    & \quad=:T_{1}+T_{2}+T_{3}+T_{4}.
  \end{align*}
  For some $\varepsilon\in(0,\min\{1-\alpha_{1},1-\alpha_{2}\})$, we can bound with Jensen's inequality (or H{\"o}lder inequality)
  \begin{align*}
    T_{1}^{q_{1}} & =\sum_{k_{1}=0}^{\infty}\Big(\sum_{k_{2}=0}^{\infty}\Big(\sum_{j_{1}=-1}^{k_{1}}\sum_{j_{2}=-1}^{k_{2}}2^{\big((1-\alpha_{1})(j_{1}-k_{1})+(1-\alpha_{2})(j_{2}-k_{2})\big)}2^{(\alpha_{1}j_{1}+\alpha_{2}j_{2})}\|f_{j_{1},j_{2}}^{*}\|_{\bar p}\Big)^{q_{2}}\Big)^{\frac{q_{1}}{q_{2}}}\\
    & \lesssim\sum_{k_{1}=0}^{\infty}\Big(\sum_{k_{2}=0}^{\infty}\sum_{j_{1}=-1}^{k_{1}}\sum_{j_{2}=-1}^{k_{2}}2^{\big((1-\alpha_{1}-\epsilon)(j_{1}-k_{1})+(1-\alpha_{2}-\epsilon)(j_{2}-k_{2})\big)q_{2}}2^{(\alpha_{1}j_{1}+\alpha_{2}j_{2})q_{2}}\|f_{j_{1},j_{2}}^{*}\|_{\bar p}^{q_{2}}\Big)^{\frac{q_{1}}{q_{2}}}\\
    & =\sum_{k_{1}=0}^{\infty}\Big(\sum_{j_{1}=-1}^{k_{1}}2^{(1-\alpha_{1}-\epsilon)(j_{1}-k_{1})q_{2}}2^{\alpha_{1}j_{1}q_{2}}\sum_{j_{2}=0}^{\infty}\sum_{j_{2} \le k_{2}}2^{(1-\alpha_{2}-\epsilon)(j_{2}-k_{2})q_{2}}2^{\alpha_{2}j_{2}q_{2}}\|f_{j_{1},j_{2}}^{*}\|_{\bar p}^{q_{2}}\Big)^{\frac{q_{1}}{q_{2}}}\\
    & \lesssim\sum_{k_{1}=0}^{\infty}\Big(\sum_{j_{1}=-1}^{k_{1}}2^{(1-\alpha_{1}-\epsilon)(j_{1}-k_{1})q_{2}}2^{\alpha_{1}j_{1}q_{2}}\Big\|\big(2^{\alpha_{2}j_{2}}\|f_{j_{1},j_{2}}^{*}\|_{\bar p}\big)_{j_{2}}\Big\|_{\ell_{q_{2}}}^{q_{2}}\Big)^{q_{1}/q_{2}}\\
    & \lesssim\Big\|\Big(2^{\alpha_{1}j_{1}}\Big\|\big(2^{\alpha_{2}j_{2}}\|f_{j_{1},j_{2}}^{*}\|_{\bar p}\big)_{j_{2}}\Big\|_{\ell_{q_{2}}}\Big)_{j_{1}}\Big\|_{\ell_{q_{1}}}^{q_{1}}.
  \end{align*}
  Moreover, for some $\epsilon\in(0,\min\{1-\alpha_{1},\alpha_{2}\})$ we have
  \begin{align*}
    T_{2}^{q_{1}} & =\sum_{k_{1}=0}^{\infty}\Big(\sum_{k_{2}=0}^{\infty}\Big(\sum_{j_{1}=-1}^{k_{1}}\sum_{j_{2}>k_{2}}2^{\big((1-\alpha_{1})(j_{1}-k_{1})-\alpha_{2}(j_{2}-k_{2})\big)}2^{(\alpha_{1}j_{1}+\alpha_{2}j_{2})}\|f_{j_{1},j_{2}}^{*}\|_{\bar p}\Big)^{q_{2}}\Big)^{\frac{q_{1}}{q_{2}}}\\
    & \lesssim\sum_{k_{1}=0}^{\infty}\Big(\sum_{j_{1}=-1}^{k_{1}}2^{(1-\alpha_{1}-\epsilon)(j_{1}-k_{1})q_{2}}2^{\alpha_{1}j_{1}q_{2}}\sum_{j_{2}=-1}^{\infty}\sum_{k_{2}=-1}^{j_{2}-1}2^{-(\alpha_{2}-\epsilon)(j_{2}-k_{2})q_{2}}2^{\alpha_{2}j_{2}q_{2}}\|f_{j_{1},j_{2}}^{*}\|_{\bar p}^{q_{2}}\Big)^{\frac{q_{1}}{q_{2}}}\\
    & =\sum_{k_{1}=0}^{\infty}\Big(\sum_{j_{1}=-1}^{k_{1}}2^{(1-\alpha_{1}-\epsilon)(j_{1}-k_{1})q_{2}}2^{\alpha_{1}j_{1}q_{2}}\Big\|\big(2^{\alpha_{2}j_{2}}\|f_{j_{1},j_{2}}^{*}\|_{\bar p}\big)_{j_{2}}\Big\|_{\ell_{q_{2}}}^{q_{2}}\Big)^{\frac{q_{1}}{q_{2}}}\\
    & \lesssim\Big\|\Big(2^{\alpha_{1}j_{1}}\Big\|\big(2^{\alpha_{2}j_{2}}\|f_{j_{1},j_{2}}^{*}\|_{\bar p}\big)_{j_{2}}\Big\|_{\ell_{q_{2}}}\Big)_{j_{1}}\Big\|_{\ell_{q_{1}}}^{q_{1}}.
  \end{align*}
  The remaining two terms can be estimated analogously. Therefore, we get
  \begin{align*}
    &\Big(\int_{\R}|h_{1}|^{-\alpha_{1}q_{1}}\Big(\int_{\T}|h_{2}|^{-\alpha_{2}q_{2}}\sup_{|r_{1}|\le|h_{1}|,|r_{2}|\le h_{2}}\|\square_{r_{1},r_{2}}f\|_{\overline{p}}^{q_{2}}\frac{\dd h_{2}}{|h_{2}|}\Big)^{q_{1}/q_{2}}\frac{\dd h_{1}}{|h_{1}|}\Big)^{1/q_{1}} \\
    &\quad\lesssim \Big\|\Big(2^{\alpha_{1}j_{1}}\Big\|\big(2^{\alpha_{2}j_{2}}\|f_{j_{1},j_{2}}^{*}\|_{\bar p}\big)_{j_{2}}\Big\|_{\ell_{q_{2}}}\Big)_{j_{1}}\Big\|_{\ell_{q_{1}}}.
  \end{align*}
  For the increments in one direction we discretize the one-dimensional integral and apply  Lemma~\ref{lem:aux1} to find    
  \begin{align*}
    &\Big(\int_{\R} h_1^{-\alpha_1 q_1}  \sup_{|r_{1}|\le|h_{1}|}\| \delta_{h_1} f \|_{\bar{p}}^{q_1} \frac{\dd h_1 }{|h_1|} \Big)^{\frac{1}{q_1}}\\
    &\quad\lesssim \Big(\sum\limits_{k_1=0}^\infty 2^{\alpha_1 k_1 q_1} \sup\limits_{|r_1| \le 2^{-k_1}} \| \delta_{h_1} f \|_{\bar{p}}^{q_1}  \Big)^{\frac{1}{q_1}}  \\
    &\quad\lesssim  \Big(\sum\limits_{k_1=0}^\infty 2^{\alpha_1 k_1 q_1} \Big( \sum\limits_{j_1=-1}^\infty  \sum\limits_{j_2=-1}^\infty \min(1,2^{(j_{1}-k_{1})}) \|f_{j_{1},j_{2}}^{*}\|_{\bar p} \Big)^{q_1} \Big)^{\frac{1}{q_1}} \\
    &\quad\lesssim  \Big(\sum\limits_{k_1=0}^\infty 2^{\alpha_1 k_1 q_1} \Big( \sum\limits_{j_1=-1}^{k_1}  \sum\limits_{j_2=-1}^\infty 2^{(j_{1}-k_{1})} \|f_{j_{1},j_{2}}^{*}\|_{\bar p} \Big)^{q_1} \Big)^{\frac{1}{q_1}} \\
    &\quad\quad  +  \Big(\sum\limits_{k_1=0}^\infty 2^{\alpha_1 k_1 q_1} \Big( \sum\limits_{j_1=k_1+1}^\infty  \sum\limits_{j_2=-1}^\infty \|f_{j_{1},j_{2}}^{*}\|_{\bar p} \Big)^{q_1} \Big)^{\frac{1}{q_1}} .
  \end{align*}
  In order to treat the first term we fix some \(\epsilon \in (0, 1 - \alpha_1 )\). Then, we get
  \begin{align*}
    & \Big(\sum\limits_{k_1=0}^\infty 2^{\alpha_1 k_1 q_1} \Big(\sum\limits_{j_2=-1}^\infty \sum\limits_{j_1=-1}^{k_1}   2^{(j_{1}-k_{1})} \|f_{j_{1},j_{2}}^{*}\|_{\bar p} \Big)^{q_1} \Big)^{\frac{1}{q_1}} \\
    &\quad\le \,  \Big(\sum\limits_{k_1=0}^\infty 2^{\alpha_1 k_1 q_1} \Big(  \sum\limits_{j_2=-1}^\infty \Big( \sum\limits_{j_1=-1}^{k_1} 2^{j_2 \alpha_2} 2^{(j_{1}-k_{1})} \|f_{j_{1},j_{2}}^{*}\|_{\bar p} \Big)^{q_2} \Big)^{\frac{q_1}{q_2}} \Big)^{\frac{1}{q_1}}
    \Big( \sum\limits_{j_2=-1}^\infty 2^{-\alpha_2 \frac{q_2}{q_2-1} j_2} \Big)^{\frac{q_2-1}{q_2}} \\
    &\quad\lesssim  \,  \Big(\sum\limits_{k_1=0}^\infty  \Big( \sum\limits_{j_1=-1}^{k_1} 2^{\alpha_1 k_1}  2^{(j_{1}-k_{1})} \Big( \sum\limits_{j_2=-1}^\infty \Big(  2^{j_2 \alpha_2}  \|f_{j_{1},j_{2}}^{*}\|_{\bar p} \Big)^{q_2} \Big)^{\frac{1}{q_2}} \Big)^{q_1} \Big)^{\frac{1}{q_1}}  \\
    &\quad=  \,  \Big(\sum\limits_{k_1=0}^\infty  \Big( \sum\limits_{j_1=-1}^{k_1} 2^{\alpha_1 j_1}  2^{(1-\alpha_1-\epsilon)(j_{1}-k_{1})} 2^{\epsilon (j_{1}-k_{1})} \Big( \sum\limits_{j_2=-1}^\infty \Big(  2^{j_2 \alpha_2}  \|f_{j_{1},j_{2}}^{*}\|_{\bar p} \Big)^{q_2} \Big)^{\frac{1}{q_2}} \Big)^{q_1} \Big)^{\frac{1}{q_1}}  \\
    &\quad\lesssim \,  \Big(\sum\limits_{k_1=0}^\infty   \sum\limits_{j_1=-1}^{k_1} 2^{\alpha_1 j_1 q_1 }  2^{(1-\alpha_1-\epsilon)(j_{1}-k_{1}) q_1} \Big( \sum\limits_{j_2=-1}^\infty \Big(  2^{j_2 \alpha_2}  \|f_{j_{1},j_{2}}^{*}\|_{\bar p} \Big)^{q_2} \Big)^{\frac{q_1}{q_2}}  \Big)^{\frac{1}{q_1}}  \\
    &\quad=  \,  \Big(\sum\limits_{j_1=-1}^\infty   2^{\alpha_1 j_1 q_1 }   \Big( \sum\limits_{j_2=-1}^\infty \Big(  2^{j_2 \alpha_2}  \|f_{j_{1},j_{2}}^{*}\|_{\bar p} \Big)^{q_2} \Big)^{\frac{q_1}{q_2}}  \Big)^{\frac{1}{q_1}}
    \Big( \sum\limits_{k_1=j_1}^{\infty} 2^{- (1-\alpha_1-\epsilon) k_1-j_1 q_1} \Big)^{\frac{1}{q_1}} \\
    &\quad\lesssim  \, \Big(\sum\limits_{j_1=-1}^\infty   2^{\alpha_1 j_1 q_1 }  \Big( \sum\limits_{j_2=-1}^\infty \Big(  2^{j_2 \alpha_2}  \|f_{j_{1},j_{2}}^{*}\|_{\bar p} \Big)^{q_2} \Big)^{\frac{q_1}{q_2}}  \Big)^{\frac{1}{q_1}} ,
  \end{align*}
  where we used H{\"o}lder's inequality in the first and forth step and the triangle inequality in the second step. Similar, for some \(\epsilon \in (0 , \alpha_1)\) we can estimate the second term as follows
  \begin{align*}
    &  \Big(\sum\limits_{k_1=0}^\infty 2^{\alpha_1 k_1 q_1} \Big( \sum\limits_{j_1=k_1+1}^\infty   \sum\limits_{j_2=-1}^\infty \|f_{j_{1},j_{2}}^{*}\|_{\bar p} \Big)^{q_1} \Big)^{\frac{1}{q_1}}  \\
    &\quad\lesssim \,  \Big(\sum\limits_{k_1=0}^\infty 2^{\alpha_1 k_1 q_1} \Big(  \sum\limits_{j_2=-1}^\infty \Big(\sum\limits_{j_1=k_1+1}^\infty 2^{j_2 \alpha_2} \|f_{j_{1},j_{2}}^{*}\|_{\bar p} \Big)^{q_2} \Big)^{\frac{q_1}{q_2}} \Big)^{\frac{1}{q_1}} \\
    &\quad\le \,  \Big(\sum\limits_{k_1=0}^\infty  \Big( \sum\limits_{j_1=k_1+1}^\infty 2^{\alpha_1 k_1 } \Big( \sum\limits_{j_2=-1}^\infty  2^{j_2 \alpha_2 q_2} \|f_{j_{1},j_{2}}^{*}\|_{\bar p}^{q_2} \Big)^{\frac{1}{q_2}} \Big)^{q_1} \Big)^{\frac{1}{q_1}}  \\
    &\quad=  \,  \Big(\sum\limits_{k_1=0}^\infty  \Big( \sum\limits_{j_1=k_1+1}^\infty 2^{\alpha_1 j_1}  2^{ - (\alpha_1-\epsilon) (j_1-k_1) } 2^{-\epsilon(j_1-k_1) } \Big( \sum\limits_{j_2=-1}^\infty  2^{j_2 \alpha_2 q_2} \|f_{j_{1},j_{2}}^{*}\|_{\bar p}^{q_2} \Big)^{\frac{1}{q_2}} \Big)^{q_1} \Big)^{\frac{1}{q_1}}  \\
    &\quad\lesssim \,  \Big(\sum\limits_{k_1=0}^\infty   \sum\limits_{j_1=k_1+1}^\infty 2^{\alpha_1 j_1 q_1}  2^{ - (\alpha_1-\epsilon) (j_1-k_1)q_1 } \Big( \sum\limits_{j_2=-1}^\infty  2^{j_2 \alpha_2 q_2} \|f_{j_{1},j_{2}}^{*}\|_{\bar p}^{q_2} \Big)^{\frac{q_1}{q_2}}  \Big)^{\frac{1}{q_1}}  \\
    &\quad\le  \,  \Big(\sum\limits_{j_1=-1}^\infty   2^{\alpha_1 j_1 q_1}   \Big( \sum\limits_{j_2=-1}^\infty  2^{j_2 \alpha_2 q_2} \|f_{j_{1},j_{2}}^{*}\|_{\bar p}^{q_2} \Big)^{\frac{q_1}{q_2}}  \Big)^{\frac{1}{q_1}} \Big( \sum\limits_{k_1=0}^\infty    2^{ - (\alpha_1-\epsilon) k_1 q_1 } \Big)^{\frac{1}{q_1}}  \\
    &\quad\lesssim  \,  \Big(\sum\limits_{j_1=-1}^\infty   2^{\alpha_1 j_1 q_1}   \Big( \sum\limits_{j_2=-1}^\infty  2^{j_2 \alpha_2 q_2} \|f_{j_{1},j_{2}}^{*}\|_{\bar p}^{q_2} \Big)^{\frac{q_1}{q_2}}  \Big)^{\frac{1}{q_1}},
  \end{align*}
  where we used H{\"o}lder's inequality in the first and forth step and the triangle inequality in the second step. Hence, we arrive at
  \begin{align*}
    \Big(\int_{\R} h_1^{-\alpha_1 q_1}  \| \delta_{h_1} f \|_{\bar{p}}^{q_1} \frac{\dd h_1 }{|h_1|} \Big)^{\frac{1}{q_1}}
    \lesssim  \Big\|\Big(2^{\alpha_{1}j_{1}}\Big\|\big(2^{\alpha_{2}j_{2}}\|f_{j_{1},j_{2}}^{*}\|_{\bar p}\big)_{j_{2}}\Big\|_{\ell_{q_{2}}}\Big)_{j_{1}}\Big\|_{\ell_{q_{1}}}.
  \end{align*}
  Performing the analogously steps with the integrand \(\delta_{h_2}f\) and combining these estimations with Corollary~\ref{cor:aux2}, we conclude
  \begin{align*}
    \|f\|_{\bar{\alpha},\bar{p},\bar{q}}^{(1)} & \lesssim\|f\|_{\bar{p}}+\Big\|\Big(2^{\alpha_{1}j_{1}}\Big\|\big(2^{\alpha_{2}j_{2}}\|f_{j_{1},j_{2}}^{*}\|_{\bar p}\big)_{j_{2}}\Big\|_{\ell_{q_{2}}}\Big)_{j_{1}}\Big\|_{\ell_{q_{1}}}\\
    & \lesssim\Big\|\Big(2^{\alpha_{1}j_{1}}\Big\|\big(2^{\alpha_{2}j_{2}}\|\Delta_{j_{1},j_{2}}f\|_{\bar p}\big)_{j_{2}}\Big\|_{\ell_{q_{2}}}\Big)_{j_{1}}\Big\|_{\ell_{q_{1}}}=\|f\|_{\bar{\alpha},\bar{p},\bar{q}}.
  \end{align*}

  \emph{Step~2:} We prove $\|f\|_{\overline{\alpha},\overline{p},\overline{q}}\lesssim\|f\|_{\overline{\alpha},\overline{p},\overline{q}}^{(1)}$, which shows together with Step~1 the equivalence of $\|f\|_{\overline{\alpha},\overline{p},\overline{q}}$ and $\|f\|_{\overline{\alpha},\overline{p},\overline{q}}^{(1)}.$

  Let $(g_{j_{1},j_{2}})_{j_{1},j_{2}\ge1}\subseteq L^{\overline{p}}(\R \times \T)$ be a sequence of functions fulfilling $\supp\F g_{j_{1},j_{2}}\subseteq\{\xi\in\R^{2}:|\xi_{1}|\le2^{j_{1}}\; \textrm{and} \; | \xi_{2}|\le2^{j_{2}}\}$ and $f=\lim_{j_{1}\wedge j_{2}\to\infty}g_{j_{1},j_{2}}$ in $\mathcal{S}^{\prime}(\R^{2})$ as constructed in Lemma~\ref{lem:aux3}.
  
  Setting $g_{j_{1},j_{2}}=0$ if $j_{1}=0$ or $j_{2}=0$ we find, for any $J_{1},J_{2}\ge 1$,
  \begin{align*}
    \sum_{j_{1}=1}^{J_{1}}\sum_{j_{2}=1}^{J_{2}}(g_{j_{1},j_{2}}-g_{j_{1}-1,j_{2}}-g_{j_{1},j_{2}-1}+g_{j_{1}-1,j_{2}-1})
    &= \sum_{j_{1}=1}^{J_{1}} (g_{j_{1},J_{2}}-g_{j_{1},0} - g_{j_{1}-1,J_{2}} +g_{j_{1}-1,0} ) \\
    &=\sum_{j_{1}=1}^{J_{1}} ( g_{j_{1},J_{2}} - g_{j_{1}-1,J_{2}})
    = g_{J_{1},J_{2}} .
  \end{align*}
  Consequently, we can write $f$ as a telescope sum
  \begin{align*}
    f &=\sum_{j_{1}=1}^{\infty}\sum_{j_{2}=1}^{\infty}(g_{j_{1},j_{2}}-g_{j_{1}-1,j_{2}}-g_{j_{1},j_{2}-1}+g_{j_{1}-1,j_{2}-1})
    =\sum_{j_{1}=1}^{\infty}\sum_{j_{2}=1}^{\infty} g_{j_{1},j_{2}}^{\square}
  \end{align*}
  with
  \begin{equation*}
    g_{j_{1},j_{2}}^{\square}:=g_{j_{1},j_{2}}-g_{j_{1}-1,j_{2}}-g_{j_{1},j_{2}-1}+g_{j_{1}-1,j_{2}-1} .
  \end{equation*}
  Next, let \(k_1,k_2 \ge -1\) be the indices from the Littlewood--Paley characterization. Then, the convergence of the telescope sum in \(\mathcal{S}^{\prime}(\R^{2})\) implies
  \begin{equation*}
    \|\Delta_{k_{1},k_{2}}f\|_{\bar{p}}  =\Big\|\sum_{j_{1}=1}^{\infty}\sum_{j_{2}=1}^{\infty} \Delta_{k_{1},k_{2}}\big(g_{j_{1},j_{2}}^{\square}\big)\Big\|_{\bar{p}}
    \le \sum_{j_{1}=1}^{\infty}\sum_{j_{2}=1}^{\infty} \big\|\Delta_{k_{1},k_{2}}\big(g_{j_{1},j_{2}}^{\square}\big)\big\|_{\bar{p}}.
  \end{equation*}
  As previously, \(g_{j_{1},j_{2}}^{\square} \in L^{\overline{p}}(\R \times \T) \) implies the periodicity in the second component of the Littlewood--Paley block \(\Delta_{k_{1},k_{2}}\big(g_{j_{1},j_{2}}^{\square}\big)\) and therefore the \(L^{\overline{p}}\)-norm is well-defined. 
  Now, the idea is to split the double sum into four terms
  \begin{align*}
    & \sum_{j_{1}=1}^{\infty}\sum_{j_{2}=1}^{\infty} \big\|\Delta_{k_{1},k_{2}} g_{j_{1},j_{2}}^{\square}\big\|  _{\bar{p}} \\
    &\quad=  \,  \big\|\Delta_{k_{1},k_{2}} g_{1,1}^{\square} \big\|_{\bar{p}} + \sum_{j_{1}=2}^{\infty}  \big\|\Delta_{k_{1},k_{2}}  g_{j_{1},1}^{\square} \big\|_{\bar{p}} + \sum_{j_{2}=2}^{\infty}   \big\|\Delta_{k_{1},k_{2}} g_{1,j_{2}}^{\square} \big\|_{\bar{p}} +  \sum_{j_{1}=2}^{\infty}\sum_{j_{2}=2}^{\infty}  \big\|\Delta_{k_{1},k_{2}}  g_{j_{1},j_{2}}^{\square} \big\|_{\bar{p}}  \\
    &\quad= \,  \big\|\Delta_{k_{1},k_{2}} g_{1,1} \big \|_{\bar{p}}  + \sum_{j_{1}=2}^{\infty} \big\|\Delta_{k_{1},k_{2}} (g_{j_1,1}-g_{j_1-1,1}) \big\|_{\bar{p}} \\
    &\qquad + \sum_{j_{2}=2}^{\infty} \big\|\Delta_{k_{1},k_{2}} (g_{1,j_{2}}-g_{1,j_2-1}) \big\|_{\bar{p}}  +  \sum_{j_{1}=2}^{\infty}\sum_{j_{2}=2}^{\infty} \big\|\Delta_{k_{1},k_{2}} g_{j_{1},j_{2}}^{\square} \big\|_{\bar{p}}
  \end{align*}
  and to analyse each one individually. More precisely, we want to estimate them in the \(\ell^{q_{1}}(\ell^{q_{2}})\) norm. By no surprise each term is going to correspond to one term in \(\|f\|_{\overline{\alpha},\overline{p},\overline{q}}^{(1)}\). Before we start with each term, let us make the following observation:

  For some \(\tilde{f}  \in L^{\overline{p}}(\R \times \T) \) we have
  \begin{align}\label{eq: triangle_young}
    \big\|\Delta_{k_{1},k_{2}}\big(\tilde{f}  \big)\big\|_{\bar{p}} & =\big\|\F^{-1}\big[(\rho_{k_{1}}\otimes\rho_{k_{2}}\big]\ast \tilde{f}  \big\|_{\bar{p}} \nonumber \\
    & \le \big\|\F^{-1}\big[(\rho_{k_{1}}\otimes\rho_{k_{2}})\big]\big\|_{L^1(\R^2)}\big\| \tilde{f}  \big\|_{\bar{p}} 
     =\|\F^{-1}[\rho]\|_{L^1(\R)}^{2}\big\| \tilde{f}  \big\|_{\bar{p}}  
     \lesssim \big\| \tilde{f}  \big\|_{\bar{p}},
  \end{align}
  where we used Young inequality for $L^{\overline{p}}(\R \times \T)$ (Lemma~\ref{lem: anisotropic_inequalities}) in the second step.

  Let us now start with the first term. By Lemma~\ref{lem:aux3} we have $\supp\F g_{1,1} \subseteq Q_{2} $, where \(Q_2\) is a cube with side length \(2\) centered at the origin. Therefore, by the support of the Littlewood--Paley decomposition we have
  \begin{equation*}
    \big\|\Delta_{k_{1},k_{2}} g_{1,1} \big \|_{\bar{p}}
    = \big\| \F^{-1}( (\rho_{k_1} \otimes \rho_{k_2}) \F g_{1,1} ) \big \|_{\bar{p}}
    = 0 \quad \quad \textrm{for} \; k_1 \ge 2 \; \mathrm{or} \; k_2 \ge 2.
  \end{equation*}
  Together with \eqref{eq: triangle_young} and an application of Young's inequality for $L^{\overline{p}}(\R \times \T)$ we find
  \begin{align}\label{eq: lipal_self_estimate}
    &\Big( \sum\limits_{k_1=-1}^\infty 2^{\alpha_1 q_1 k_1} \Big( \sum\limits_{k_2=-1}^\infty  2^{\alpha_2 q_2 k_2} \big\|\Delta_{k_{1},k_{2}} g_{1,1} \big \|_{\bar{p}}^{q_2} \Big)^{\frac{q_1}{q_2}} \Big)^{\frac{1}{q_1}}\nonumber \\
    &\quad= \Big( \sum\limits_{k_1=-1}^1 2^{\alpha_1 q_1 k_1} \Big( \sum\limits_{k_2=-1}^1 2^{\alpha_2 q_2 k_2 } \big\|\Delta_{k_{1},k_{2}} g_{1,1} \big\|_{\bar{p}}^{q_2} \Big)^{\frac{q_1}{q_2}} \Big)^{\frac{1}{q_1}}  \nonumber \\
    &\quad \lesssim \big\|g_{1,1} \big\|_{\bar{p}} 
    = \big\|(K_1 \otimes K_1) \ast f \big\|_{\bar{p}}  
    \le \Big( \int_{\R}| K_1(h)| \dd h \Big)^2  \big\|f \big\|_{\bar{p}} 
    \lesssim \big\| f \big\|_{\bar{p}} .
  \end{align}
  This completes the treatment of the first term.

  For the second term we note that by the support of the Littlewood--Paley decomposition and \(\supp\F g_{j_1,1} \subseteq \{  \xi\in\R^{2}: |\xi_1| \le 2^{j_1} \; \mathrm{and} \; |\xi_2| \le 2 \}\) we have
  \begin{equation*}
    \big\|\Delta_{k_{1},k_{2}} (g_{j_1,1}-g_{j-1,1})  \big \|_{\bar{p}} = \big\| \F^{-1}( (\rho_{k_1} \otimes \rho_{k_2}) \F ((g_{j_1,1}-g_{j-1,1})) ) \big \|_{\bar{p}} = 0
  \end{equation*}
  for $k_1 \ge j_1 + 2$ or $k_2 \ge 2$. Consequently, we have
  \begin{align*}
    &\Big( \sum\limits_{k_1=-1}^\infty 2^{\alpha_1 q_1 k_1} \Big( \sum\limits_{k_2=-1}^\infty  2^{\alpha_2 q_2 k_2 } \Big( \sum\limits_{j_1=2}^\infty \big\|\Delta_{k_{1},k_{2}}  (g_{j_1,1}-g_{j-1,1}) \big \|_{\bar{p}} \Big)^{q_2} \Big)^{\frac{q_1}{q_2}} \Big)^{\frac{1}{q_1}}  \\
    &\quad\le \,\Big( \sum\limits_{k_1=-1}^\infty 2^{\alpha_1 q_1 k_1} \Big( \sum\limits_{j_1=2}^\infty  \Big(  \sum\limits_{k_2=-1}^{1}  2^{\alpha_2 q_2 k_2} \1_{\{k_1 -1 \le j_1\}} \big\|\Delta_{k_{1},k_{2}}  (g_{j_1,1}-g_{j-1,1}) \big\|_{\bar{p}}^{q_2}  \Big)^{\frac{1}{q_2}} \Big)^{q_1} \Big)^{\frac{1}{q_1}} \\
    &\quad= \,   \Big( \sum\limits_{k_1=-1}^{\infty} 2^{\alpha_1 q_1 k_1}  \Big( \sum\limits_{j_1=\max(k_1-1,2)}^\infty  \Big(  \sum\limits_{k_2=-1}^{1 }  2^{\alpha_2 q_2 k_2 } \big\|\Delta_{k_{1},k_{2}}  (g_{j_1,1}-g_{j-1,1}) \big \|_{\bar{p}}^{q_2}  \Big)^{\frac{1}{q_2}} \Big)^{q_1} \Big)^{\frac{1}{q_1}} \\
    &\quad \lesssim  \,  \Big( \sum\limits_{k_1=-1}^{\infty} 2^{\alpha_1 q_1 k_1} \Big( \sum\limits_{j_1=\max(k_1-1,2)}^\infty   \Big( \sum\limits_{k_2=-1}^{1}   2^{\alpha_2 q_2 k_2 } \big\|  g_{j_1,1}-g_{j-1,1} \big \|_{\bar{p}}^{q_2}  \Big)^{\frac{1}{q_2}} \Big)^{q_1} \Big)^{\frac{1}{q_1}} \\
    &\quad \lesssim \,  \Big( \sum\limits_{k_1=-1}^{\infty} 2^{\alpha_1 q_1 k_1} \Big( \sum\limits_{j_1=\max(k_1-1,2)}^\infty    \sup\limits_{|r_1|\le 2^{-j_{1}}} \|\delta_{r_1}f\|_{\bar{p}}  \Big)^{q_1} \Big)^{\frac{1}{q_1}} \\
    &\quad \le  \,  \Big( \sum\limits_{k_1=-1}^{\infty} 2^{\alpha_1 q_1 k_1} \Big( \sum\limits_{j_1=k_1-1}^\infty    \sup\limits_{|r_1|\le 2^{-j_{1}}} \|\delta_{r_1}f\|_{\bar{p}}  \Big)^{q_1} \Big)^{\frac{1}{q_1}} \\
    &\quad\le  \,  \Big( \sum\limits_{k_1=-1}^{\infty} 2^{\alpha_1 q_1 k_1} \Big( \sum\limits_{j_1=1}^\infty \,  \sup\limits_{\substack{|r_1| \le   2^{-(j_{1}+k_1-2)}}} \|\delta_{r_1}f\|_{\bar{p}}  \Big)^{q_1} \Big)^{\frac{1}{q_1}} \\
    &\quad\lesssim \,  \sum\limits_{j_1=1}^\infty 2^{ - \alpha_1 j_1} \Big( \sum\limits_{k_1=-1}^{\infty} 2^{\alpha_1 q_1 (k_1-2+j_1)}   \sup\limits_{\substack{|r_1| \le   2^{-(j_{1}+k_1-2)}}} \|\delta_{r_1}f\|_{\bar{p}}^{q_1} \Big)^{\frac{1}{q_1}} \\
    &\quad \le \,  \sum\limits_{j_1=1}^\infty 2^{ - \alpha_1 j_1} \Big( \sum\limits_{k_1=-2}^{\infty} 2^{\alpha_1 q_1 k_1}  \sup\limits_{|r_1| \le   2^{-k_{1}}} \|\delta_{r_1}f\|_{\bar{p}}^{q_1} \Big)^{\frac{1}{q_1}} \\
    &\quad \lesssim \,  \Big( \sum\limits_{k_1=-2}^{\infty} 2^{\alpha_1 q_1 k_1}   \sup\limits_{|r_1| \le   2^{-k_{1}}} \|\delta_{r_1}f\|_{\bar{p}}^{q_1} \Big)^{\frac{1}{q_1}} ,
  \end{align*}
  where we used \eqref{eq: triangle_young} in the third step and \eqref{eq: delta_1_l_p_estimate} in the fourth step. We also utilized multiple times the Fourier support of the functions as well as multiple index shifts.

  Now, a Riemann sum argument reveals that
  \begin{align}\label{eq: lipal_delta_1_estimate}
    & \Big( \sum\limits_{k_1=-1}^\infty 2^{\alpha_1 q_1 k_1} \Big( \sum\limits_{k_2=-1}^\infty  2^{\alpha_2 q_2 k_2 }  \sum\limits_{j_1=2}^\infty \big\|\Delta_{k_{1},k_{2}}  (g_{j_1,1}-g_{j-1,1}) \big \|_{\bar{p}}^{q_2}  \Big)^{\frac{q_2}{q_1}} \Big)^{\frac{1}{q_1}} \nonumber  \\
    &\quad\lesssim  \,
    \Big( \int_{\R} h_{1}^{-\alpha_{1}q_{1}}  \sup\limits_{|r_1| \le |h_1|} \|\delta_{r_{1}}^1 f\|_{{\overline{p}}}^{q_{1}}\frac{\dd h_{1}}{|h_{1}|}  \Big)^{\frac{1}{q_1}} .
  \end{align}
  Going through the same steps with the third term we obtain
  \begin{align}\label{eq: lipal_delta_2_estimate}
    & \Big( \sum\limits_{k_1=-1}^\infty 2^{\alpha_1 q_1 k_1} \Big( \sum\limits_{k_2=-1}^\infty  2^{\alpha_2 q_2 k_2 }  \sum\limits_{j_2=2}^\infty \big\|\Delta_{k_{1},k_{2}}  (g_{1,j_2}-g_{1,j_2-1}) \big \|_{\bar{p}}^{q_2}  \Big)^{\frac{q_2}{q_1}} \Big)^{\frac{1}{q_1}} \nonumber \\
     &\quad \lesssim  \,
    \Big( \int_{\T} h_{2}^{-\alpha_{2}q_{2}}  \sup\limits_{|r_2| \le |h_2|} \|\delta_{r_{2}}^1 f\|_{{\overline{p}}}^{q_{2}}\frac{\dd h_{2}}{|h_{2}|}  \Big)^{\frac{1}{q_2}} .
  \end{align}
  Hence, it remains to estimate the last term in our telescope sum decomposition. Again, we observe that by the support of the Littlewood--Paley decomposition and \(\supp \F g_{j_1,j_2} \subseteq  \{  \xi\in\R^{2}: |\xi_1| \le 2^{j_1} \; \mathrm{and} \;  |\xi_2| \le 2^{j_2} \}\) we have
  \begin{equation*}
    \big\|\Delta_{k_{1},k_{2}} g_{j_1,j_2}^\square \big \|_{\bar{p}} = \big\| \F^{-1}( (\rho_{k_1} \otimes \rho_{k_2}) \F g_{j_1,j_2}^\square ) \big \|_{\bar{p}} = 0 \quad \quad \textrm{for} \; k_1 \ge j_1+2 \; \mathrm{or} \; k_2 \ge j_2+2.
  \end{equation*}
  Therefore, we obtain 
  \begin{align*}
    &  \Big( \sum\limits_{k_1=-1}^\infty 2^{\alpha_1 q_1 k_1} \Big( \sum\limits_{k_2=-1}^\infty  2^{\alpha_2 q_2 k_2 } \Big( \sum\limits_{j_1=2}^\infty \sum\limits_{j_2=2}^\infty \big\|\Delta_{k_{1},k_{2}}
    g_{j_1,j_2}^\square \big \|_{\bar{p}} \Big)^{q_2}  \Big)^{\frac{q_1}{q_2}} \Big)^{\frac{1}{q_1}}  \\
    &\quad= \, \Big( \sum\limits_{k_1=-1}^\infty 2^{\alpha_1 q_1 k_1} \Big( \sum\limits_{k_2=-1}^\infty  2^{\alpha_2 q_2 k_2 } \Big( \sum\limits_{j_1=2}^\infty \sum\limits_{j_2=2}^\infty \1_{\{ k_1-1 \le j_1 \}} \1_{\{ k_2 -1 \le j_2 \}} \big\|\Delta_{k_{1},k_{2}}
    g_{j_1,j_2}^\square \big \|_{\bar{p}} \Big)^{q_2}  \Big)^{\frac{q_1}{q_2}} \Big)^{\frac{1}{q_1}}  \\
    &\quad\lesssim \,  \Big( \sum\limits_{k_1=-1}^\infty 2^{\alpha_1 q_1 k_1} \Big( \sum\limits_{k_2=-1}^\infty  2^{\alpha_2 q_2 k_2 } \Big( \sum\limits_{j_1=2}^\infty \sum\limits_{j_2=2}^\infty \1_{\{ k_1-1 \le j_1 \}} \1_{\{ k_2 -1 \le j_2 \}} \big\|
    g_{j_1,j_2}^\square \big \|_{\bar{p}} \Big)^{q_2}  \Big)^{\frac{q_1}{q_2}} \Big)^{\frac{1}{q_1}}  \\
    &\quad\lesssim  \,  \Big( \sum\limits_{k_1=-1}^\infty 2^{\alpha_1 q_1 k_1} \Big( \sum\limits_{k_2=-1}^\infty  2^{\alpha_2 q_2 k_2 } \Big( \sum\limits_{j_1=2}^\infty \sum\limits_{j_2=2}^\infty \1_{\{ k_1-1 \le j_1 \}} \1_{\{ k_2 -1 \le j_2 \}}  \sup\limits_{|r_i|\le 2^{-j_{i}}} \|\square_{r}f\|_{\bar{p}}  \Big)^{q_2}  \Big)^{\frac{q_1}{q_2}} \Big)^{\frac{1}{q_1}}  \\
    &\quad \le \,  \Big( \sum\limits_{k_1=-1}^\infty 2^{\alpha_1 q_1 k_1} \Big( \sum\limits_{j_1=k_1-1}^\infty   \Big( \sum\limits_{k_2=-1}^\infty  2^{\alpha_2 q_2 k_2 } \Big(  \sum\limits_{j_2=k_2-1}^\infty   \sup\limits_{|r_i|\le 2^{-j_{i}}} \|\square_{r}f\|_{\bar{p}}  \Big)^{q_2}  \Big)^{\frac{1}{q_2}} \Big)^{q_1} \Big)^{\frac{1}{q_1}}  \\
    &\quad\le  \,  \Big( \sum\limits_{k_1=-1}^\infty 2^{\alpha_1 q_1 k_1} \Big( \sum\limits_{j_1=1}^\infty   \Big( \sum\limits_{k_2=-1}^\infty  2^{\alpha_2 q_2 k_2 } \Big(  \sum\limits_{j_2=1}^\infty   \sup\limits_{|r_i|\le 2^{-(j_{i}+k_{i}-2)}} \|\square_{r}f\|_{\bar{p}}  \Big)^{q_2}  \Big)^{\frac{1}{q_2}} \Big)^{q_1} \Big)^{\frac{1}{q_1}}  \\
    &\quad\le \,  \sum\limits_{j_1=1}^\infty  \sum\limits_{j_2=1}^\infty  \Big( \sum\limits_{k_1=-1}^\infty 2^{\alpha_1 q_1 k_1}  \Big( \sum\limits_{k_2=-1}^\infty  2^{\alpha_2 q_2 k_2 } \Big(     \sup\limits_{|r_i|\le 2^{-(j_{i}+k_{i}-2)}} \|\square_{r}f\|_{\bar{p}}  \Big)^{q_2}  \Big)^{\frac{q_1}{q_2}}  \Big)^{\frac{1}{q_1}}  \\
    &\quad\lesssim  \,  \sum\limits_{j_1=1}^\infty  \sum\limits_{j_2=1}^\infty  2^{-\alpha_1  j_1} 2^{-\alpha_2j_2 } \Big( \sum\limits_{k_1=-1}^\infty 2^{\alpha_1 q_1 (j_1+k_1-2)}  \Big( \sum\limits_{k_2=-1}^\infty  2^{\alpha_2 q_2 (j_2+k_2-2)} \\
    &\qquad \times \Big(  \sup\limits_{|r_i|\le 2^{-(j_{i}+k_{i}-2)}} \|\square_{r}f\|_{\bar{p}}  \Big)^{q_2}  \Big)^{\frac{q_1}{q_2}}  \Big)^{\frac{1}{q_1}}  \\
    &\quad \le \,  \sum\limits_{j_1=1}^\infty  \sum\limits_{j_2=1}^\infty  2^{-\alpha_1  j_1} 2^{-\alpha_2j_2 } \Big( \sum\limits_{k_1=-2}^\infty 2^{\alpha_1 q_1 k_1}  \Big( \sum\limits_{k_2=-2}^\infty  2^{\alpha_2 q_2 k_2} \Big(  \sup\limits_{|r_i|\le 2^{-k_{i}}} \|\square_{r}f\|_{\bar{p}}  \Big)^{q_2}  \Big)^{\frac{q_1}{q_2}}  \Big)^{\frac{1}{q_1}}  \\
    &\quad\lesssim  \,  \Big( \sum\limits_{k_1=-2}^\infty 2^{\alpha_1 q_1 k_1}  \Big( \sum\limits_{k_2=-2}^\infty  2^{\alpha_2 q_2 k_2} \Big(  \sup\limits_{|r_i|\le 2^{-k_{i}}} \|\square_{r}f\|_{\bar{p}}  \Big)^{q_2}  \Big)^{\frac{q_1}{q_2}}  \Big)^{\frac{1}{q_1}} ,
  \end{align*}
  where we used \eqref{eq: triangle_young} in the second step and \eqref{eq: square_l_p_estimate} in the third step. After that, we used multiple index shifts as well as the triangle inequality for the \(\ell^{q_1},\ell^{q_2}\)-norms in the seventh step. Employing a Riemann sum arguments yields
  \begin{align}\label{eq: lipal_square_estimate}
    & \Big( \sum\limits_{k_1=-1}^\infty 2^{\alpha_1 q_1 k_1} \Big( \sum\limits_{k_2=-1}^\infty  2^{\alpha_2 q_2 k_2 } \Big( \sum\limits_{j_1=2}^\infty \sum\limits_{j_2=2}^\infty \big\|\Delta_{k_{1},k_{2}}
    g_{j_1,j_2}^\square \big \|_{\bar{p}} \Big)^{q_2}  \Big)^{\frac{q_1}{q_2}} \Big)^{\frac{1}{q_1}}   \nonumber \\
    &\quad \lesssim  \Big(\int_{\R}|h_{1}|^{-\alpha_{1}q_{1}}\Big(\int_{\T}|h_{2}|^{-\alpha_{2}q_{2}}\sup_{|r_{1}|\le|h_{1}|,|r_{2}|\le h_{2}}\|\square_{r_{1},r_{2}}f\|_{\overline{p}}^{q_{2}}\frac{\dd h_{2}}{|h_{2}|}\Big)^{q_{1}/q_{2}}\frac{\dd h_{1}}{|h_{1}|}\Big)^{1/q_{1}}.
  \end{align}

  Putting \eqref{eq: lipal_self_estimate}, \eqref{eq: lipal_delta_1_estimate}, \eqref{eq: lipal_delta_2_estimate} and \eqref{eq: lipal_square_estimate} together we finally obtain
  \begin{align*}
    &\Big\|\Big(2^{\alpha_{1}k_1}\big\|\big(2^{\alpha_{2}k_2}\|\Delta_{k_1,k_2}f\|_{\bar p}\big)_{k_1\ge-1}\big\|_{\ell^{q_{2}}}\Big)_{k_2\ge-1}\Big\|_{\ell^{q_{1}}} \\
    &\quad= \, \Big\|\Big(2^{\alpha_{1}k_1}\big\|\big(2^{\alpha_{2}k_2}\|\Delta_{k_1,k_2} \Big( \sum_{j_{1}=1}^{\infty}\sum_{j_{2}=1}^{\infty} g_{j_{1},j_{2}}^{\square} \Big) \|_{\bar p}\big)_{k_1\ge-1}\big\|_{\ell^{q_{2}}}\Big)_{k_2\ge-1}\Big\|_{\ell^{q_{1}}} \\
    &\quad\le \, \Big\|\Big(2^{\alpha_{1}k_1}\big\|\big(2^{\alpha_{2}k_2}   \big\|\Delta_{k_{1},k_{2}} g_{1,1} \big \|_{\bar{p}} \big)_{k_1\ge-1}\big\|_{\ell^{q_{2}}}\Big)_{k_2\ge-1}\Big\|_{\ell^{q_{1}}}   \\
    &\qquad+ \Big\|\Big(2^{\alpha_{1}k_1}\big\|\big(2^{\alpha_{2}k_2} \sum_{j_{1}=2}^{\infty} \big\|\Delta_{k_{1},k_{2}} (g_{j_1,1}-g_{j_1-1,1}) \big\|_{\bar{p}} \big)_{k_1\ge-1}\big\|_{\ell^{q_{2}}}\Big)_{k_2\ge-1}\Big\|_{\ell^{q_{1}}}  \\
    &\qquad+ \Big\|\Big(2^{\alpha_{1}k_1}\big\|\big(2^{\alpha_{2}k_2} \sum_{j_{2}=2}^{\infty} \big\|\Delta_{k_{1},k_{2}} (g_{1,j_{2}}-g_{1,j_2-1}) \big\|_{\bar{p}} \big)_{k_1\ge-1}\big\|_{\ell^{q_{2}}}\Big)_{k_2\ge-1}\Big\|_{\ell^{q_{1}}}  \\
    & \qquad+ \Big\|\Big(2^{\alpha_{1}k_1}\big\|\big(2^{\alpha_{2}k_2} \sum_{j_{1}=2}^{\infty}\sum_{j_{2}=2}^{\infty} \big\|\Delta_{k_{1},k_{2}} g_{j_{1},j_{2}}^{\square} \big\|_{\bar{p}} \big)_{k_1\ge-1}\big\|_{\ell^{q_{2}}}\Big)_{k_2\ge-1}\Big\|_{\ell^{q_{1}}} \\
    &\quad \lesssim \,  \|f\|_{\overline{\alpha},\overline{p},\overline{q}}^{(1)}.
  \end{align*}

  \emph{Step~3}: We show the equivalence of $\|f\|_{\overline{\alpha},\overline{p},\overline{q}}^{(1)}$ and $\|f\|_{\overline{\alpha},\overline{p},\overline{q}}^{(2)}.$ Since $\|f\|_{\overline{\alpha},\overline{p},\overline{q}}^{(2)}\lesssim\|f\|_{\overline{\alpha},\overline{p},\overline{q}}^{(1)}$ is obvious, it remains to verify $\|f\|_{\overline{\alpha},\overline{p},\overline{q}}^{(1)}\lesssim\|f\|_{\overline{\alpha},\overline{p},\overline{q}}^{(2)}$. We start with by estimating the square integrands in the norms \(\|f\|_{\overline{\alpha},\overline{p},\overline{q}}^{(1)}\) and \(\|f\|_{\overline{\alpha},\overline{p},\overline{q}}^{(2)}\). Moreover, we notice that the only difficulty happens for \(h_{i}\) near zero. Therefore, we are only interested in a small rectangle near zero. Let $r_{1},r_{2}\in\R$ satisfy $|h_{i}|\leq|r_{i}|\leq2|h_{i}|$ for $i=1,2$. We write the rectangular increments as
  \begin{equation*}
    \square_{(r_{1},r_{2})}f(x_{1},x_{2})
    =\square_{(r_{1},\frac{r_{2}}{|r_{2}|}|h_{2}|)}f(x_{1},x_{2})+\square_{(r_{1},r_{2}-\frac{r_{2}}{|r_{2}|}|h_{2}|)}f\Big(x_{1},x_{2}+\frac{r_{2}}{|r_{2}|}|h_{2}|\Big).
  \end{equation*}
  Owing to $|r_{2}-\frac{|h_{2}|}{|r_{2}|}r_{2}|=|r_{2}|(1-\frac{|h_{2}|}{|r_{2}|})=|r_{2}|-|h_{2}|\le|h_{2}|$, we obtain
  \begin{equation*}
    \|\square_{(r_{1},r_{2})}f\|_{\overline{p}}\leq\sup_{\tau:|\tau|=1}\|\square_{(r_{1},\tau h_{2})}f\|_{\overline{p}}+\sup_{|s_{2}|\leq|h_{2}|}\|\square_{(r_{1},s_{2})}f\|_{\overline{p}}.
  \end{equation*}
  The above ineqaulity is obviously true for \(r_2 \le |h_2|\), too. Therefore, 
  \begin{equation*}
    \sup_{|r_{1}|\leq|h_{1}|,|r_{2}|\leq2|h_{2}|}\|\square_{(r_{1},r_{2})}f\|_{\overline{p}}\leq\sup_{|r_{1}|\leq|h_{1}|,|\tau|=1}\|\square_{(r_{1},\tau h_{2})}f\|_{\overline{p}}+\sup_{|r_{1}|\leq|h_{1}|,|r_{2}|\leq|h_{2}|}\|\square_{(r_{1},r_{2})}f\|_{\overline{p}}.
  \end{equation*}
  Starting with $q_{1},q_{2}<\infty$, integrating both sides leads to
  \begin{align*}
    & \Big(\int_{\R}|h_{1}|^{-\alpha_{1}q_{1}}\Big(\int_{\T}|h_{2}|^{-\alpha_{2}q_{2}}\sup_{|r_{1}|\leq|h_{1}|,|r_{2}|\leq2|h_{2}|}\|\square_{(r_{1},r_{2})}f\|_{\overline{p}}^{q_{2}}\frac{\dd h_{2}}{|h_{2}|}\Big)^{q_{1}/q_{2}}\frac{\dd h_{1}}{|h_{1}|}\Big)^{1/q_{1}}\\
    & \quad\leq\Big(\int_{\R}|h_{1}|^{-\alpha_{1}q_{1}}\Big(\int_{\T}|h_{2}|^{-\alpha_{2}q_{2}}\sup_{|r_{1}|\leq|h_{1}|,|\tau|=1}\|\square_{(r_{1},\tau h_{2})}f\|_{\overline{p}}\frac{\dd h_{2}}{|h_{2}|}\Big)^{q_{1}/q_{2}}\frac{\dd h_{1}}{|h_{1}|}\Big)^{1/q_{1}}\\
    & \qquad+\Big(\int_{\R}|h_{1}|^{-\alpha_{1}q_{1}}\Big(\int_{\T}|h_{2}|^{-\alpha_{2}q_{2}}\sup_{|r_{1}|\leq|h_{1}|,|r_{2}|\leq|h_{2}|}\|\square_{(r_{1},r_{2})}f\|_{\overline{p}}\frac{\dd h_{2}}{|h_{2}|}\Big)^{q_{1}/q_{2}}\frac{\dd h_{1}}{|h_{1}|}\Big)^{1/q_{1}}.
  \end{align*}
  Hence a change of variable $2h_{2}\to h_{2}$ reveals
  \begin{align*}
    & 2^{\alpha_{2}}\Big(\int_{\R}|h_{1}|^{-\alpha_{1}q_{1}}\Big(\int_{\T}|h_{2}|^{-\alpha_{2}q_{2}}\sup_{|r_{1}|\leq|h_{1}|,|r_{2}|\leq h_{2}}\|\square_{(r_{1},r_{2})}f\|_{\overline{p}}^{q_{2}}\frac{\dd h_{2}}{|h_{2}|}\Big)^{q_{1}/q_{2}}\frac{\dd h_{1}}{|h_{1}|}\Big)^{1/q_{1}}\\
    & \quad\leq\Big(\int_{\R}|h_{1}|^{-\alpha_{1}q_{1}}\Big(\int_{\T}|h_{2}|^{-\alpha_{2}q_{2}}\sup_{|r_{1}|\leq|h_{1}|,|\tau|=1}\|\square_{(r_{1},\tau h_{2})}f\|_{\overline{p}}\frac{\dd h_{2}}{|h_{2}|}\Big)^{q_{1}/q_{2}}\frac{\dd h_{1}}{|h_{1}|}\Big)^{1/q_{1}}\\
    & \qquad+\Big(\int_{\R}|h_{1}|^{-\alpha_{1}q_{1}}\Big(\int_{\T}|h_{2}|^{-\alpha_{2}q_{2}}\sup_{|r_{1}|\leq|h_{1}|,|r_{2}|\leq|h_{2}|}\|\square_{(r_{1},r_{2})}f\|_{\overline{p}}\frac{\dd h_{2}}{|h_{2}|}\Big)^{q_{1}/q_{2}}\frac{\dd h_{1}}{|h_{1}|}\Big)^{1/q_{1}}.
  \end{align*}
  Therefore,
  \begin{align}\label{eq:estimate second component}
    & \Big(\int_{\R}|h_{1}|^{-\alpha_{1}q_{1}}\Big(\int_{\T}|h_{2}|^{-\alpha_{2}q_{2}}\sup_{|r_{1}|\leq|h_{1}|,|r_{2}|\leq|h_{2}|}\|\square_{(r_{1},r_{2})}f\|_{\overline{p}}^{q_{2}}\frac{\dd h_{2}}{|h_{2}|}\Big)^{q_{1}/q_{2}}\frac{\dd h_{1}}{|h_{1}|}\Big)^{1/q_{1}}\nonumber \\
    & \quad\leq\frac{1}{2^{\alpha_{2}}-1}\Big(\int_{\R}|h_{1}|^{-\alpha_{1}q_{1}}\Big(\int_{\T}|h_{2}|^{-\alpha_{2}q_{2}}\sup_{|r_{1}|\leq|h_{1}|,|\tau|=1}\|\square_{(r_{1},\tau h_{2})}f\|_{\overline{p}}\frac{\dd h_{2}}{|h_{2}|}\Big)^{q_{1}/q_{2}}\frac{\dd h_{1}}{|h_{1}|}\Big)^{1/q_{1}}\nonumber \\
    & \quad\le\frac{2}{2^{\alpha_{2}}-1}\Big(\int_{\R}|h_{1}|^{-\alpha_{1}q_{1}}\Big(\int_{\T}|h_{2}|^{-\alpha_{2}q_{2}}\sup_{|r_{1}|\leq|h_{1}|}\|\square_{(r_{1},h_{2})}f\|_{\overline{p}}\frac{\dd h_{2}}{|h_{2}|}\Big)^{q_{1}/q_{2}}\frac{\dd h_{1}}{|h_{1}|}\Big)^{1/q_{1}},
  \end{align}
  where the last estimated follows from $\sup_{|\tau|=1}\|\square_{(r_{1},\tau h_{2})}f\|_{\overline{p}}\le\|\square_{(r_{1},h_{2})}f\|_{\overline{p}}+\|\square_{(r_{1},-h_{2})}f\|_{\overline{p}}$ and substituting $-h_{2}\to h_{2}$ for the second term. Proceeding similarly in the first coordinate, we deduce from
  \begin{equation*}
    \sup_{|r_{1}|\leq2|h_{1}|}\|\square_{(r_{1},h_{2})}f\|_{\overline{p}}\le\sup_{|\tau|=1}\|\square_{(\tau h_{1},h_{2})}f\|_{\overline{p}}+\sup_{|r_{1}|\leq|h_{1}|}\|\square_{(r_{1},h_{2})}\|_{\overline{p}}
  \end{equation*}
  that
  \begin{align}\label{eq:estimate first component}
    & \Big(\int_{\R}|h_{1}|^{-\alpha_{1}q_{1}}\Big(\int_{\T}|h_{2}|^{-\alpha_{2}q_{2}}\sup_{|r_{1}|\leq h_{1}}\|\square_{(r_{1},h_{2})}f\|_{\overline{p}}\frac{\dd h_{2}}{|h_{2}|}\Big)^{q_{1}/q_{2}}\frac{\dd h_{1}}{|h_{1}|}\Big)^{1/q_{1}}\nonumber \\
    & \quad\leq\frac{2}{2^{\alpha_{1}}-1}\Big(\int_{\R}|h_{1}|^{-\alpha_{1}q_{1}}\Big(\int_{\T}|h_{2}|^{-\alpha_{2}q_{2}}\|\square_{(h_{1},h_{2})}f\|_{\overline{p}}\frac{\dd h_{2}}{|h_{2}|}\Big)^{q_{1}/q_{2}}\frac{\dd h_{1}}{|h_{1}|}\Big)^{1/q_{1}}.
  \end{align}
  Combing \eqref{eq:estimate second component} and \eqref{eq:estimate first component}, we find
  \begin{align*}
    & \Big(\int_{\R}|h_{1}|^{-\alpha_{1}q_{1}}\Big(\int_{\T}|h_{2}|^{-\alpha_{2}q_{2}}\sup_{|r_{1}|\leq h_{1},|r_{2}|\leq h_{2}}\|\square_{(r_{1},r_{2})}f\|_{\overline{p}}^{q_{2}}\frac{\dd h_{2}}{|h_{2}|}\Big)^{q_{1}/q_{2}}\frac{\dd h_{1}}{|h_{1}|}\Big)^{1/q_{1}}\\
    & \quad\leq\frac{2}{2^{\alpha_{1}}-1}\frac{2}{2^{\alpha_{2}}-1}\Big(\int_{\R}|h_{1}|^{-\alpha_{1}q_{1}}\Big(\int_{\T}|h_{2}|^{-\alpha_{2}q_{2}}\|\square_{(h_{1},h_{2})}f\|_{\overline{p}}\frac{\dd h_{2}}{|h_{2}|}\Big)^{q_{1}/q_{2}}\frac{\dd h_{1}}{|h_{1}|}\Big)^{1/q_{1}}\\
    & \quad=\frac{2}{2^{\alpha_{1}}-1}\frac{2}{2^{\alpha_{2}}-1}\Big\|\frac{\|\square_{(h_{1},h_{2})}f\|_{\overline{p}}}{|h_{1}|^{\alpha_{1}+1/q_{1}}|h_{2}|^{\alpha_{2}+1/q_{2}}}\Big\|_{(q_{1},q_{2})}.
  \end{align*}
  Furthermore, we obtain the corresponding result for $q_{1}=\infty$ or $q_{2}=\infty$ by canonical modifications of the above formulas.

  Moreover, the directional integrands \(\delta_{h_1}f, \; \delta_{h_2}f\) in the norms \(\|f\|_{\overline{\alpha},\overline{p},\overline{q}}^{(1)}, \; \|f\|_{\overline{\alpha},\overline{p},\overline{q}}^{(2)}\) behave as in the isotropic case. Therefore, one can repeat \cite[Theorem 2.5.12, Step 3]{Triebel1978} in a similar fashion as above to obtain the corresponding estimates. This implies the equivalence of \(\|f\|_{\overline{\alpha},\overline{p},\overline{q}}^{(1)}, \; \|f\|_{\overline{\alpha},\overline{p},\overline{q}}^{(2)}\) and the theorem is proven.
\end{proof}

\subsection{Besov spaces with dominating mixed smoothness on compact domains}

Functions spaces on domains play a crucial role, e.g., in the study of partial differential equations and of stochastic processes, cf. Section~\ref{sec:regularity of random fields}. In view of the difference characterizations of Besov spaces with dominating mixed smoothness, presented in Theorem~\ref{thm:characteristation}, there is a natural localized version of the Besov norm with dominating mixed smoothness as well. We define for the domain $D:=[0,T]\times\mathbb T$ the local Besov norms with dominating mixed smoothness on $D$ by
\begin{align*}
  \|f\|_{D,\overline{\alpha},\overline{p},\overline{q}}^{(2)}
  &:=\|f\|_{D,\overline{p}}+\Big(\int_{0}^{1}h_{1}^{-\alpha_{1}q_{1}}\Big(\int_{0}^{1}h_{2}^{-\alpha_{2}q_{2}}\big(\|\square_{(h_{1},h_{2})}f\|_{D-h_1,\overline{p}}\big)^{q_{2}}\frac{\dd h_{2}}{|h_{2}|}\Big)^{q_{1}/q_{2}}\frac{\dd h_{1}}{|h_{1}|}\Big)^{\frac{1}{q_{1}}}\\
  &\quad+\Big(\int_{0}^{1}h_{2}^{-\alpha_{2}q_{2}}\big(\|\delta_{h_{2}} f\|_{D,\overline{p}}\big)^{q_{2}}\frac{\dd h_{2}}{|h_{2}|}\Big)^{\frac{1}{q_{2}}}
  +\Big( \int_{0}^{1}h_{1}^{-\alpha_{1}q_{1}}\big(\|\delta_{h_{1}} f\|_{D-h_1,\overline{p}}\big)^{q_{1}}\frac{\dd h_{1}}{|h_{1}|}\Big)^{\frac{1}{q_{1}}},
\end{align*}
where
\begin{equation*}
  \|f\|_{D,\bar{p}} :=\Big(\int_{0}^{T}\Big(\int_{\mathbb T}|f(x_{1},x_{2})|^{p_{2}}\dd x_{2}\Big)^{p_{1}/p_{2}}\dd x_{1}\Big)^{\frac{1}{p_{1}}}
\end{equation*}
and $D-h_1:=[0,\min\{(T-h_1),0\}]\times\mathbb T\}]$. Following e.g. \cite[Section~4.3.2]{Triebel1978}, we introduce
\begin{equation*}
  B_{\overline{p},\overline{q}}^{\overline{\alpha}}(D):= \{ f\in B_{\overline{p},\overline{q}}^{\overline{\alpha}}(\R\times\mathbb T) \,:\, \supp f \subseteq D \}
\end{equation*}
and
\begin{equation*}
  \mathcal{S}_{\pi_2}^{\prime}(D) := \{ f\in \mathcal{S}_{\pi_2}^{\prime}(\R^2)
  \,:\, \supp f \subseteq D \}.
\end{equation*}
The following result is a direct consequence of Theorem~\ref{thm:characteristation} and verifies that $\|f\|_{D,\overline{\alpha},\overline{p},\overline{q}}^{(2)}$ is indeed a localized version of $\|f\|_{\overline{\alpha},\overline{p},\overline{q}}$:

\begin{corollary}
  Let $ \bar p,\bar q\in[1,\infty]^2$ and $\bar \alpha\in(0,1)^2$. Then, $\|\cdot\|_{D,\overline{\alpha},\overline{p},\overline{q}}^{(2)}$ and $\|\cdot\|_{\overline{\alpha},\overline{p},\overline{q}}$ are equivalent norms in $B_{\overline{p},\overline{q}}^{\overline{\alpha}}(D)$, that is,
  \begin{equation*}
    \|f\|_{\overline{\alpha},\overline{p},\overline{q}}
    \lesssim\|f\|_{D,\overline{\alpha},\overline{p},\overline{q}}^{(2)}
    \lesssim\|f\|_{\overline{\alpha},\overline{p},\overline{q}},\quad\text{for }f\in\mathcal{S}^{\prime}(D) \cap L^{\overline{p}}(\R^2).
  \end{equation*}
\end{corollary}

\begin{remark}
  Another way to define a local version of the Besov spaces with dominating mixed smoothness is to introduce the  norm
  \begin{equation*}
    \|f\|_{D,\overline{\alpha},\overline{p},\overline{q}}
    := \inf \{ \|g\|_{\overline{\alpha},\overline{p},\overline{q}}\,:\, g\in \mathcal{S}^\prime (\R^2) \text{ s.t. }
    g = f \text{ on }D \}.
  \end{equation*}
  Unfortunately, to the best of our knowledge in the case with dominating mixed smoothness not many results are known regarding the above norm. However, there are many results regarding these type of norm in the isotropic case, see for instance \cite{Triebel1978} and \cite{Rychkov1999}. The alternative characterization of the Besov spaces (Theorem~\ref{thm:characteristation}) could be a crucial tool to obtain the corresponding results in the case with dominating mixed smoothness, which appears to be an interesting question for future research.
\end{remark}

The global Besov norm with dominating mixed smoothness and its corresponding local version are additionally related by following multiplier theorem.

\begin{proposition}
  Let $f\in\mathcal{S}_{\pi_2}'(\R^{2})$ with $\|f\|_{D,\overline{\alpha},\overline{p},\overline{q}}^{(2)}<\infty$ for some $D:=[0,T]\times \mathbb T$, where $T\in(0,\infty)$, $\bar{\alpha}\in(0,1 )^{2},\bar{p},\bar{q}\in[1,\infty]^{2}$. Then for any $\phi\in\mathcal{S}(\R)$ satisfying $\supp\phi\subseteq[0,T]$ we have $(\phi\otimes 1) f\in B_{\bar{p},\bar{q}}^{\bar{\alpha}}(\R^{2})$ and
  \begin{equation*}
    \|(\phi\otimes 1) f\|_{\overline{\alpha},\overline{p},\overline{q}}
    \lesssim\|\phi\|_{C^1} \|f\|_{D,\overline{\alpha},\overline{p},\overline{q}}^{(2)}
  \end{equation*}
  for $\|\phi\|_{C^1}:=\|\phi\|_\infty+\|\phi'\|_\infty$.
\end{proposition}

\begin{proof}
  For the $L^{\bar{p}}$-norm the support assumption on $\phi$ immediately implies
  \begin{align*}
    \|(\phi\otimes 1) f\|_{\overline{p}} & =\Big(\int_{\R}\Big(\int_{\T}|\phi(x_{1})f(x_{1},x_{2})|^{p_{2}}\dd x_{2}\Big)^{p_{1}/p_{2}}\dd x_{1}\Big)^{1/p_{1}}\\
    & =\Big(\int_{0}^{T}\Big(\int_{\T}|\phi(x_{1})|^{p_{2}}|f(x_{1},x_{2})|^{p_{2}}\dd x_{2}\Big)^{p_{1}/p_{2}}\dd x_{1}\Big)^{1/p_{1}}\\
    & \le\|\phi\|_{\infty}\|f\|_{D,\bar{p}}.
  \end{align*}
  Next we have to estimate $\|\square_{h_{1},h_{2}}((\phi\otimes 1) f)\|_{D-h_{1},\overline{p}}$ for any $h_{1}\in[0,1]^{2}$. We can write with directional increments
  \begin{align*}
    \square_{h_{1},h_{2}}((\phi\otimes 1) f)(x_{1},x_{2}) & =\delta_{h_{1}}^{1}\delta_{h_{2}}^{2}((\phi\otimes 1) f)(x_{1},x_{2})\\
    & =\delta_{h_{1}}^{1}\big(\phi(x_{1})\delta_{h_{2}}^{2}f(x_{1},x_{2})\big)\\
    & =\phi(x_{1}+h_{1})\square_{h_{1},h_{2}}f(x_{1},x_{2})+\delta_{h_{1}}^{1}\phi(x_{1})\delta_{h_{2}}^{2}f(x_{1},x_{2}).
  \end{align*}
  The first term we can estimate as follows
  \begin{equation*}
    |\phi(x_{1}+h_{1})\square_{h_{1},h_{2}}f(x_{1},x_{2})|
    \leq  \| \phi  \|_{L^\infty} | \square_{h_{1},h_{2}}f(x_{1},x_{2})| .
  \end{equation*}
  For the second term we use the mean value theorem in order to find
  \begin{equation*}
    |\delta_{h_{1}}^{1}\phi(x_{1}) \delta_{h_{2}}^{2}f(x_{1},x_{2})|
    \le \|\phi' \|_{L^\infty} h_1 | \delta_{h_{2}}^{2}f(x_{1},x_{2})|.
  \end{equation*}
  Consequently, we obtain
  \begin{align*}
    &\|  \square_{h_{1},h_{2}}((\phi\otimes 1) f)(x_{1},x_{2}) \|_{D-h_{1}, \overline{p}}\\
    &\quad \leq   \; \;   \| \phi  \|_{L^\infty} \|  \| \square_{h_{1},h_{2}}f \|_{D-h_{1}, \overline{p}}  + h_1 \|\phi' \|_{L^\infty}  \| \delta_{h_2}^2 f   \|_{D-h_{1}, \overline{p}}.
  \end{align*}
  Now, we have to integrate over \(h_1,h_2\) with the associated weights/powers (see definition of \(\|f\|_{D,\overline{\alpha},\overline{p},\overline{q}}^{(2)} \)). We notice that \(h_{1}^{-\alpha_{1}q_{1}} \frac{h_1^{q_{1}}}{|h_{1}|} \) is integrable over \([0,1]\) for \(i=1,2\) since \(1-\alpha_{i} > 0 \). Hence, we have
  \begin{align*}
    &\int_{0}^{1}h_{1}^{-\alpha_{1}q_{1}}\Big(\int_{0}^{1}h_{2}^{-\alpha_{2}q_{2}}\big(\|\square_{h_{1},h_{2}}((\phi\otimes 1) f)\|_{D-h_1,\overline{p}}\big)^{q_{2}}\frac{\dd h_{2}}{|h_{2}|}\Big)^{q_{1}/q_{2}}\frac{\dd h_{1}}{|h_{1}|}  \\
    &\quad\lesssim \| \phi\|_{C^1}^{q_1}   \Big( \int_{0}^{1}h_{1}^{-\alpha_{1}q_{1}}\Big(\int_{0}^{1}h_{2}^{-\alpha_{2}q_{2}}\big(\|\square_{h_{1},h_{2}} f\|_{D-h_1,\overline{p}}\big)^{q_{2}}\frac{\dd h_{2}}{|h_{2}|}\Big)^{q_{1}/q_{2}}\frac{\dd h_{1}}{|h_{1}|}  \\
    &\quad\quad + \Big(\int_{0}^{1}h_{2}^{-\alpha_{2}q_{2}}\big(\|\delta_{h_{2}} f\|_{D,\overline{p}}\big)^{q_{2}}\frac{\dd h_{2}}{|h_{2}|}\Big)^{q_{1}/q_{2}} \Big)\\
    &\quad\lesssim \| \phi\|_{C^1}^{q_1}  \Big(  \|f\|_{D,\overline{\alpha},\overline{p},\overline{q}}^{(2)}\Big)^{q_1}.
  \end{align*}
  Since
  \[\|\delta_{h_2}((\phi\otimes 1)f) \|_{D,\overline{p}}=\|\phi\delta_{h_2}f \|_{D,\overline{p}}\le \|\phi\|_\infty\|\delta_{h_2}f \|_{D,\overline{p}},\]
  it remains to examine the term \(\|\delta_{h_1}((\phi\otimes 1) f) \|_{D-h_{1},\overline{p}}\). Using the mean-value theorem we find
  \begin{align*}
    |\delta_{h_1}((\phi\otimes1) f)| \le |f(x_1+h_1,x_2) \delta_{h_1}\phi | + |\phi \delta_{h_1} f  |
    \le \| \phi'  \|_{L^\infty} | h_1 |f(x_1+h_1,x_2)|   + \| \phi  \|_{L^\infty} | \delta_{h_1} f  |.
  \end{align*}
  Consequently, multiplying with \(h_1^{-q_1\alpha_1}\) and integrating over \(h_1\) we find 
  \begin{align*}
    &  \int\limits_0^1 h_1^{-\alpha_1 q_1} \|\delta_{h_1}(\phi f)\|_{D-h_1,\overline{p}}^{q_1} \frac{\dd h_1 }{|h_1|} \\
    &\quad\le \,
    \| \phi'  \|_{L^\infty}^{q_1}   \int\limits_0^1 h_1^{(1-\alpha_1) q_1} \|f(\cdot+h_1,x_2)\|_{D-h_1,\overline{p}}^{q_1} \frac{\dd h_1 }{|h_1|}
    +  \| \phi  \|_{L^\infty}  ^{q_1}   \int\limits_0^1 h_1^{-\alpha_1 q_1} \|\delta_{h_1} f \|_{D-h_1,\overline{p}}^{q_1} \frac{\dd h_1 }{|h_1|} \\
    &\quad\lesssim \,  \| \phi\|_{C^1}^{q_1}   \Big(\|f\|_{D, \overline{p}}^{q_1} +  \int\limits_0^1 h_1^{-\alpha_1 q_1} \|\delta_{h_1} f \|_{D-h_1,\overline{p}}^{q_1} \frac{\dd h_1 }{|h_1|} \Big)
  \end{align*}
  where the first integral over \(h_1\) is finite since \(\|f(\cdot+h_1,x_2)\|_{D-h_1,\overline{p}} \le \|f\|_{D,\overline{p}}\) and \(h_1^{(1-\alpha_1) q_1 -1} \) is integrable over \([0,1]\). Combining this with the previous estimates we establish the result.
\end{proof}

\section{Regularity of random fields}\label{sec:regularity of random fields}

While having access to various equivalent characterizations of Besov spaces comes with many advantages, in this section, we want to briefly discuss one in the context of probability. As a rather straightforward consequence of the presented characterisation of Besov spaces with mixed dominating smoothness by differences (Theorem~\ref{thm:characteristation}), we shall derive a version of Kolmogorov's continuity criterion for the Besov regularity with dominating mixed smoothness of random fields. Recall, that the Kolmogorov's continuity criterion is a standard tool to verify the sample path regularity of stochastic processes, see e.g. \cite{Stroock2006}. Moreover, let us remark that the study of Besov, Sobolev and H{\"o}lder regularity of stochastic processes and random fields is a classical topic of on-going interest, see, e.g., \cite{Ciesielski1993,Roynette1993,Schilling1997,Schilling1998,Fageot2017,Hummel2021} and the references therein as well as \cite{Henderson2024} for a more detailed discussion on related works.

\smallskip

Let $(\Omega,\mathcal{F},\mathbb{P})$ be a probability space with expectation operator $\mathbb{E}$, $T\in [1,\infty)$, and $(\R^{m},|\cdot|)$ be the Euclidean space. Furthermore, let $X\colon[0,T]\times\T\times\Omega\to\R^{m}$ be a continuous random field which is periodic in the second argument, that is, for each $x_1,x_2 \in [0,T]\times\T $, $X(x_1,x_2)$ is a real valued random variable and the mapping $(x_1,x_2) \mapsto X(x_1,x_2)(\omega)$ is continuous for almost all $\omega \in \Omega$.

\begin{proposition}\label{prop:Kolmogorov}
  Let $X\colon[0,T]\times\T\times\Omega\to\R^{m}$ be a continuous random field. If $\overline{p}=(p_1,p_2),\overline{q}=(q_1,q_2)\in [1,\infty)^2$ with $q_1\leq q_2 \leq p_1\leq p_2$ and $\overline{\alpha}=(\alpha_1,\alpha_2)\in (0,1)^2$ such that
  \begin{equation}\label{eq:Kolmogorov}
    \mathbb{E}\big[|\square_{(h_1,h_2)}X(x_1,x_2)|^{p_2}\big]
    \lesssim h_{1}^{(1+\alpha_1 q_1)\frac{p_2}{q_1}}h_{2}^{(1+\alpha_2 q_2)\frac{p_2}{q_2}},
  \end{equation}
  \begin{equation*}
    \mathbb{E}\big[|\delta_{h_1}X(x)|^{p_2}\big]\lesssim h_{1}^{(1+\alpha_1 q_1)\frac{p_2}{q_1}}
    \quad \text{and}\quad
    \mathbb{E}\big[|\delta_{h_2}X(x)|^{p_2}\big]\lesssim  h_{2}^{(1+\alpha_2 q_2)\frac{p_2}{q_2}},
  \end{equation*}
  for all $x_1\in[0,T-h_1]$, $x_2\in\T$ and $h_{1},h_{2}\in[0,T]$, then $\|X\|_{[0,T]\times \T,\overline{\alpha},\overline{p},\overline{q}}^{(2)} <\infty$ almost surely, where $\overline{\alpha}=(\alpha_1,\alpha_2)$.
\end{proposition}

\begin{proof}
  Using Fubini's theorem, Jensen's inequality the assumption of the proposition, we get for $D=[0,T]\times\T$
  \begin{align*}
    &\mathbb{E}\Big[\int_{0}^{1}h_{1}^{-\alpha_{1}q_1}\Big( \int_{0}^{1}h_{2}^{-\alpha_{2}q_2}\big(\|\square_{(h_{1},h_{2})}X\|_{D-h_1,\overline{p}}\big)^{q_2}\frac{\dd h_{2}}{|h_{2}|}\Big)^{\frac{q_1}{q_2}}\frac{\dd h_{1}}{|h_{1}|}\Big]\\
    & \,\, \leq \int_{0}^{1}\Big(\int_{0}^{1}\Big(\int_0^{T-h_1}\Big(\int_{\T} \mathbb{E}\big[ |\square_{(h_{1},h_{2})}X(x_1,x_2)|^{p_2}\big]\dd x_2\Big)^{\frac{p_1}{p_2}} \dd x_1\Big)^{\frac{q_2}{p_1}} \frac{\dd h_{2}}{h_{2}^{1+\alpha_{2}q_2}}\Big)^{\frac{q_1}{q_2}}\frac{\dd h_{1}}{h_{1}^{1+\alpha_{1}q_1}}\\
    &\,\,\lesssim \int_{0}^{1}\Big(\int_{0}^{1}h_{1}^{(1+\alpha_{1}q_1)\frac{p_2}{q_1}\frac{p_1}{p_2}\frac{q_2}{p_1}}  h_{2}^{(1+\alpha_{2}q_2)\frac{p_2}{q_2}\frac{p_1}{p_2}\frac{q_2}{p_1}} \frac{\dd h_{2}}{h_{2}^{1+\alpha_{2}q_2}}\Big)^{\frac{q_1}{q_2}}\frac{\dd h_{1}}{h_{1}^{1+\alpha_{1}q_1}}
    \,\, \lesssim 1,\\
    &\mathbb{E}\Big[\int_{0}^{1}h_{1}^{-\alpha_{1}q_1}\big(\|\delta_{h_{1}}X\|_{D-h_1,\overline{p}}\big)^{q_1}\frac{\dd h_{1}}{|h_{1}|}\Big]\\
    & \,\, = \int_{0}^{1}h_{1}^{-\alpha_{1}q_1}
    \Big( \int_0^{T-h_1}\Big( \int_{\T} \mathbb{E}\big[ |\delta_{h_{1}}X(x_1,x_2)|^{p_2}\big]\dd x_2\Big)^{\frac{p_1}{p_2}}\dd x_1\Big)^{\frac{q_1}{p_1}} \frac{\dd h_{1}}{h_{1}}
    \,\,\lesssim 1,
    \intertext{and}
    &\mathbb{E}\Big[\int_{0}^{1}h_{2}^{-\alpha_{2}q_2}\big(\|\delta_{h_{2}}X\|_{D,\overline{p}}\big)^{q_2}\frac{\dd h_{2}}{|h_{2}|}\Big]\\
    & \,\, = \int_{0}^{1}h_{2}^{-\alpha_{2}q_2}
    \Big( \int_0^{T}\Big( \int_{\T} \mathbb{E}\big[ |\delta_{h_{2}}X(x_1,x_2)|^{p_2}\big]\dd x_2\Big)^{\frac{p_1}{p_2}}\dd x_1\Big)^{\frac{q_2}{p_1}} \frac{\dd h_{2}}{|h_{2}|}
    \,\,\lesssim 1.
  \end{align*}
  Furthermore, since $X$ is continuous on $[0,T]^2$, we have $\|X\|_{D,\overline{p}}<\infty $ almost surely.

  Hence, we get $\mathbb E[\|X\|^{(2)}_{D,\overline{\alpha},\overline{p},\overline{q}}]<\infty$ and thus $\|X\|^{(2)}_{D,\overline{\alpha},\overline{p},\overline{q}}<\infty$ almost surely.
\end{proof}

\begin{remark}
  The assumption in Proposition~\ref{prop:Kolmogorov} that the random field~$X$ is continuous can replaced by a suitable condition on the parameters $\overline{p},\overline{q},\overline{\alpha}$. Indeed, since
  \begin{equation*}
    \mathbb{E}\big[|X(x+h)-X(x)|^{p_2}\big]
    \lesssim
    \mathbb{E}\big[|\square_{h}X(x)|^{p_2}\big]
    +\mathbb{E}\big[|\delta_{h_1}X(x)|^{p_2}\big]
    +\mathbb{E}\big[|\delta_{h_2}X(x)|^{p_2}\big],
  \end{equation*}
  one can apply the classical Kolmogorov continuity theorem for random fields (see e.g. \cite[Theorem~2.3.1]{Khoshnevisan2002}) to obtain a continuous modification of~$X$ for suitable $\overline{p},\overline{q},\overline{\alpha}$.
\end{remark}

\begin{remark}
  The special case $p_1=p_2=q_1=q_2=\infty$, which corresponds to H\"older regularity with mixed dominating smoothness, is a common tool in probability theory, used, e.g., in the context of stochastic partial differentiable equations~\cite{Quer2007} and of general random fields~\cite{Hu2013}. For the special case of H{\"o}lder regularity, a Kolmogorov continuity criterion, related to Proposition~~\ref{prop:Kolmogorov}, was presented in \cite[Theorem~3.1]{Hu2013}.
\end{remark}

\begin{remark}
  Proposition~\ref{prop:Kolmogorov} holds analogously for random fields~$X$ on the domains $[0,T]\times [0,T]$ and $\mathbb{T}\times\mathbb{T}$, respectively.
\end{remark}

To conclude this section, we provide two exemplary applications of the presented Kolmogorov's continuity criterion for Besov regularity with dominating mixed smoothness.

\subsection{Exemplary application: Gaussian fields}

Gaussian fields constitute a central object in probability theory, see e.g. \cite{Khoshnevisan2002}. In the following we consider a continuous centred Gaussian random field~$W$ with covariance function
\begin{equation*}
  Q(x,y):=\mathbb{E}[W(x)W(y)]\quad \text{for }(x,y)\in [0,T]\times \T.
\end{equation*}
In particular, for each $x\in[0,T]$, $W(x)$ is a Gaussian distributed random variable with mean zero and variance $Q(x,x)$.

\begin{corollary}
  Let $p\in \N$ be even and $\bar\alpha =(\alpha_1,\alpha_2)\in (0,1)^2$. Let $W$ be continuous centred Gaussian random field such that
  \begin{equation*}
    Q(x,y) = Q_1(x_1,y_1)Q_2(x_2,y_2),\quad x,y\in [0,T]\times\T,
  \end{equation*}
  for some continuous functions $Q_1,Q_2\colon[0,T]\times\T\to \R$ satisfying
  \begin{equation*}
    |\square_{(h_i,h_i)} Q_i(x_i,x_i)|\lesssim h_i^{2(1+\alpha_i)/p}, \quad h_1,h_2\in [0,1],
    \, x_1 \in [0,T-h_1],\, x_2\in \T,
  \end{equation*}
  for $i=1,2$. Then, $\|W\|_{[0,T]\times\T,\overline{\alpha},\overline{p},\overline{p}}^{(2)} <\infty$ almost surely with $\overline{p}=(p,p)$.
\end{corollary}

\begin{proof}
  Since $W$ is a Gaussian field, we know that $\square_{h} W(x)$ is a Gaussian random variable for $h=(h_1,h_2)\in [0,1]^2$, $x_1\in [0,T-h_1]$ and $x_2\in \T$. By the equivalence of Gaussian moments, the assumption of the lemma and \cite[Proposition~4.1]{Hu2013}, we obtain
  \begin{equation*}
    \mathbb{E} [| \square_h W(x)|^p]
    = C_p \mathbb{E} [| \square_h W(x)|^2]^{p/2}
    =  C_p  \Pi_{i=1}^2|\square_{(h_i,h_i)} Q_i(x_i,x_i)|^{p/2}
    \lesssim   h_1^{1+\alpha_1} h_2^{1+\alpha_2}.
  \end{equation*}
  and
  \begin{align*}
    \mathbb{E} [| \delta_{h_1} W(x)|^p]
    &= C_p \mathbb{E} [| \delta_{h_1} W(x) |^2]^{p/2}\\
    &= C_p \mathbb{E} \big[ W^2(x_1+h_1,x_2)-2W(x_1+h_1,x_2)W(x_1,x_2)+W^2(x_1,x_2)\big]^{p/2}\\
    &= C_p |Q_2(x_2,x_2) \square_{(h_1,h_1)} Q_1(x_1,x_1)|^{p/2}
    \lesssim h_2^{1+\alpha_1}.
  \end{align*}
  for some constant $C_p$ depending only on $p$. The bound for $  \mathbb{E} [| \delta_{h_2} W(x)|^p]$ follows analogously. Hence, by Proposition~\ref{prop:Kolmogorov} we get $\|W\|_{[0,T]\times\T,\overline{\alpha},(p,p),(p,p)}^{(2)} <\infty$ almost surely.
\end{proof}

\subsection{Exemplary application: stochastic heat equation}

Random fields are often modelled by stochastic partial differential equations, see e.g. \cite{Dalang2009}. To illustrate the previous result let us characterize the regularity of the solution to the stochastic heat equation on $\R_{+}\times [-\pi,\pi]$ with Dirichlet boundary conditions:
\begin{align*}
  \begin{cases}
  \partial_{t}u(t,x) & =\Delta u(t,x)+\xi(t,x), \\
  u(t,-\pi)&=u(t,\pi)=0,\qquad t\in\R_+\\
  u(0,x) & =0,\qquad (t,x)\in \R_{+}\times\T,
  \end{cases}
\end{align*}
with space-time white noise $\xi$. The latter can be realized as time derivative $\xi=d W_t(x)$ of a cylindrical Brownian motion $W_t(x)=\sum_{j\ge1} B_j(t)\phi_j(x)$, where $B_j=(B_j(t),t\ge1)$ are independent Brownian motions on a probability space $(\Omega, \mathcal{F},\mathbb{P})$ and $(\phi_j)_{j\ge1}$ is the orthonormal basis given by the eigenfunctions of Laplace operator on $\{g\in H^2((-\pi,\pi)):g(-\pi)=g(\pi)=0\}$.

It is well known that the weak solution to stochastic heat equation is almost $1/2$-H{\"o}lder in space and almost $1/4$-H{\"o}lder regular in time. The mixed regularity allows us to characterize the interplay between these to directional regularities. To apply Proposition~\ref{prop:Kolmogorov}, we can exploit that the solution field $u$ is a centred Gaussian field admitting precise estimates of the rectangular increments, cf. \cite{HildebrandtTrabs2021}.

\begin{corollary}
  Let $T>0$, $a\in(0,1)$ and $1\le q\le p$ with even $p$. If $\bar\alpha=(\frac{a}{4}-\frac{1}{q},\frac{1-a}{2}-\frac{1}{q})\in(0,1]^2$, then $u\in B^{\bar\alpha}_{\overline{p},\overline{q}}([0,T]\times\T)$ almost surely with $\overline{p}=(p,p)$ and $\overline{q}=(q,q)$.
\end{corollary}

\begin{proof}
  Since $u$ is a Gaussian field, we know that $\square_{h} u(x)$ is a Gaussian random variable for $h=(h_1,h_2)\in [0,1]^2$, $t\in [0,T-h_1]$ and $x\in \T$. By the equivalence of Gaussian moments and \cite[Proposition~3.5]{HildebrandtTrabs2021}, we obtain
  \begin{align*}
    \mathbb{E} [|\square_h u(t,x)|^p]
    = C_p \mathbb{E} [| \square_h u(t,x)|^2]^{p/2}
    \lesssim \min\big\{\sqrt{h_1},h_2\big\}^{p/2}
    \le h_1^{ap/4}h_2^{(1-a)p/2}
  \end{align*}
  for any $a\in(0,1)$ and where $C_p>0$ is some constant depending only on $p$. Note that the upper bound in $\min\{\sqrt{h_1},h_2\}$ is sharp. We obtain \eqref{eq:Kolmogorov} for $\alpha_1=\frac{a}{4}-\frac{1}{q}$ and $\alpha_2=\frac{1-a}{2}-\frac{1}{q}$. The necessary bounds for $\mathbb{E} [| \square_{h_i} u(t,x)|^p],i=1,2,$ follow analogously using $\mathbb E[(u(t+h_1,x)-u(t,x))^2]\lesssim \sqrt{h_1}$ and $\mathbb E[(u(t,x+h_2)-u(t,x))^2]\lesssim h_2$ \cite[Lemma~5.21]{DaPrato1992}. Hence, Proposition~\ref{prop:Kolmogorov} yields $\|u\|_{[0,T]\times\T,\overline{\alpha},(p,p),(p,p)}^{(2)} <\infty$ almost surely.
\end{proof}


\newcommand{\etalchar}[1]{$^{#1}$}
\def\cprime{$'$}
\providecommand{\bysame}{\leavevmode\hbox to3em{\hrulefill}\thinspace}
\providecommand{\MR}{\relax\ifhmode\unskip\space\fi MR }
\providecommand{\MRhref}[2]{%
  \href{http://www.ams.org/mathscinet-getitem?mr=#1}{#2}
}
\providecommand{\href}[2]{#2}

\end{document}